\documentclass{amsart}
\usepackage[margin=1in]{geometry}
\usepackage{amsmath, amsthm, amssymb, color, url}
\usepackage[final]{showkeys}

\providecommand{\url}[1]{\url{#1}}

\newcommand{\CC}{\mathbb{C}}

\newcommand{\bb}{\mathfrak{b}}
\newcommand{\HH}{\mathcal{H}}
\newcommand{\ZZ}{\mathbb{Z}}

\newcommand{\ux}{\underline{x}}

\newcommand{\gt}{\mathsf{gt}}
\newcommand{\sgt}{\mathsf{sgt}}
\newcommand{\wx}{\wtilde{x}}
\newcommand{\wpsi}{\wtilde{\psi}}
\newcommand{\wD}{\wtilde{D}}
\newcommand{\wC}{\mathsf{C}}
\newcommand{\bbmu}{\boldsymbol{\mu}}
\newcommand{\wbbmu}{\widetilde{\bbmu}}

\newcommand{\wmu}{\wtilde{\mu}}
\newcommand{\wnu}{\wtilde{\nu}}
\newcommand{\wlambda}{\wtilde{\lambda}}
\newcommand{\ulambda}{\underline{\lambda}}
\newcommand{\wulambda}{\wtilde{\ulambda}}
\newcommand{\bbsigma}{\boldsymbol{\sigma}}
\newcommand{\bbeta}{\boldsymbol{\eta}}
\newcommand{\bbxi}{\boldsymbol{\xi}}
\newcommand{\blambda}{\bar{\lambda}}
\newcommand{\btau}{\bar{\tau}}
\newcommand{\bmu}{\bar{\mu}}
\newcommand{\bnu}{\bar{\nu}}

\newcommand{\wtau}{\wtilde{\tau}}
\newcommand{\utau}{\underline{\tau}}
\newcommand{\wutau}{\wtilde{\utau}}
\newcommand{\wrho}{\wtilde{\rho}}
\newcommand{\wPhi}{\wtilde{\Phi}}
\newcommand{\wPsi}{\wtilde{\Psi}}

\newcommand{\id}{\text{id}}

\newcommand{\Tr}{\text{Tr}}
\newcommand{\Sym}{\text{Sym}}

\newcommand{\wt}{\text{wt}}

\newcommand{\eps}{\varepsilon}

\newcommand{\one}{\mathbf{1}}

\newcommand{\wtilde}{\widetilde}

\newcommand{\Res}{\text{Res}}
\newcommand{\wRes}{\wtilde{\Res}}
\newcommand{\bRes}{\mathsf{Res}}

\newcommand{\End}{\text{End}}

\newcommand{\gl}{\mathfrak{gl}}

\theoremstyle{definition}
\newtheorem{theorem}{Theorem}[section]

\newtheorem{corr}[theorem]{Corollary}
\newtheorem{lemma}[theorem]{Lemma}
\newtheorem{prop}[theorem]{Proposition}
\newtheorem*{remark}{Remark}

\usepackage{hyperref}
\hypersetup{
    colorlinks,
    citecolor=blue,
    filecolor=black,
    linkcolor=black,
    urlcolor=black
}

\begin{document}

\title{A representation-theoretic proof of the branching rule for Macdonald polynomials}
\author{Yi Sun}
\date{\today}
\email{yisun@math.mit.edu}

\begin{abstract}
We give a new representation-theoretic proof of the branching rule for Macdonald polynomials using the Etingof-Kirillov Jr. expression for Macdonald polynomials as traces of intertwiners of $U_q(\gl_n)$ given in \cite{EK}.  In the Gelfand-Tsetlin basis, we show that diagonal matrix elements of such intertwiners are given by application of Macdonald's operators to a simple kernel.  An essential ingredient in the proof is a map between spherical parts of double affine Hecke algebras of different ranks based upon the Dunkl-Kasatani conjecture of \cite{Dun, Eno, ES, Kas}.
\end{abstract}

\maketitle

\tableofcontents

\section{Introduction}

The Macdonald polynomials $P_\lambda(x; q, t)$ are a two-parameter family of symmetric polynomials indexed by partitions $\lambda$ which form an orthogonal basis for the ring of symmetric functions with respect to a $(q, t)$-deformation of the standard inner product.  They were originally introduced by Macdonald (see \cite{Mac}) as a generalization of many known families of special functions, including Schur functions, Jack and Hall-Littlewood polynomials, and Heckman-Opdam hypergeometric functions.  Macdonald proved a branching rule for the $P_\lambda(x; q, t)$ and conjectured three additional symmetry, evaluation, and norm identities collectively known as Macdonald's conjectures.  These conjectures were proven by Cherednik using techniques from double affine Hecke algebras in \cite{Che2}.  Etingof and Kirillov~Jr. realized the Macdonald polynomials in \cite{EK} in terms of traces of intertwiners of the quantum group $U_q(\gl_n)$; using this interpretation, they gave new proofs of Macdonald's conjectures in \cite{EK2}.  

The purpose of this paper is to give a representation-theoretic proof and interpretation of Macdonald's branching rule from the perspective of quantum groups.  We give a new expression for diagonal matrix elements of $U_q(\gl_n)$-intertwiners in the Gelfand-Tsetlin basis as the application of Macdonald's difference operators to a simple kernel.  We then show that the resulting summation expression for $P_\lambda(x; q, t)$ becomes Macdonald's branching rule after a summation by parts procedure. A key ingredient which is of independent interest is the construction of a map $\Res_l$ between spherical parts of double affine Hecke algebras of different ranks.  Our construction makes essential use of the Dunkl-Kasatani conjecture stated in \cite{Dun, Kas} and proven in \cite{Eno, ES} and is compatible with Cherednik's $SL_2(\ZZ)$-action on spherical DAHA.

In the remainder of the introduction, we summarize our motivations, give precise statements of our results, and explain how they relate to other recent work.

\subsection{Macdonald polynomials}

Let $\rho = \left(\frac{n-1}{2}, \ldots, \frac{1 - n}{2}\right)$ and let $e_r$ denote the elementary symmetric polynomial. For a partition $\lambda$, the Macdonald polynomial $P_\lambda(x; q^2, t^2)$ is the joint polynomial eigenfunction with leading term $x^\lambda$ and eigenvalue $e_r(q^{2\lambda} t^{2\rho})$ of the operators
\[
D_{n, x}^r(q^2, t^2) = t^{r(r-n)} \sum_{|I| = r} \prod_{i \in I, j \notin I} \frac{t^2 x_i - x_j}{x_i - x_j} T_{q^2, I},
\]
where $T_{q^2, I} = \prod_{i \in I} T_{q^2, i}$ and $T_{q^2, i}f(x_1, \ldots, x_n) = f(x_1, \ldots, q^2 x_i, \ldots, x_n)$ so that we have
\[
D_{n, x}^r(q^2, t^2) P_\lambda(x, q^2, t^2) = e_r(q^{2\lambda}t^{2\rho}) P_\lambda(x, q^2, t^2).
\]
Note that our normalization of $D_{n, x}^r(q^2, t^2)$ differs from that of \cite{Mac}.  An integral signature $\lambda$ is a sequence $\lambda = (\lambda_1 \geq \cdots \geq \lambda_n)$ with $\lambda_i - \lambda_j \in \ZZ$, and it is dominant if $\lambda_i \geq \lambda_{i+1}$.  We extend the definition of Macdonald polynomials to arbitrary signatures by setting 
\[
P_{(\lambda_1 + c, \ldots, \lambda_n + c)}(x; q^2, t^2) = (x_1 \cdots x_n)^c P_\lambda(x; q^2, t^2).
\]
We say that integral signatures $\mu = (\mu_1 \geq \cdots \geq \mu_{n-1})$ and $\lambda = (\lambda_1 \geq \cdots \geq \lambda_n)$ interlace if
\[
\lambda_1 \geq \mu_1 \geq \lambda_2 \geq \cdots \geq \mu_{n-1} \geq \lambda_n.
\]
Denote interlacing by $\mu \prec \lambda$ and write $|\lambda| = \sum_i \lambda_i$. A Gelfand-Tsetlin pattern subordinate to $\lambda$ is an interlacing sequence 
\[
\bbmu = \{\mu^l\}_{1 \leq l \leq n} = \{\mu^1 \prec \mu^2 \prec \cdots \prec \mu^{n-1} \prec \mu^n = \lambda\}
\]
ending in $\lambda$.  Define the $q$-Pochhammer symbol by 
\[
(u; q)_\infty = \prod_{n \geq 0} (1 - u q^n).
\]
In \cite{Mac}, Macdonald showed that $P_\lambda(x)$ satisfies the following branching rule, which yields an explicit summation expression for $P_\lambda(x)$ over Gelfand-Tsetlin patterns subordinate to $\lambda$. 

\begin{theorem}[{\cite[VI.7.13']{Mac}}] \label{thm:branch}
The Macdonald polynomials satisfy the branching rule
\[
P_\lambda(x_1, \ldots, x_n; q, t) = \sum_{\mu \prec \lambda} \psi_{\lambda / \mu}(q, t) P_\mu(x_1, \ldots, x_{n-1}; q, t) x_n^{|\lambda| - |\mu|}
\]
where the branching coefficient is
\[
\psi_{\lambda / \mu}(q, t) = \prod_{1 \leq i \leq j \leq \ell(\mu)} \frac{(q^{\mu_i - \mu_j} t^{j - i + 1}; q)_\infty(q^{\lambda_i - \lambda_{j+1}} t^{j - i + 1}; q)_\infty (q^{\lambda_i - \mu_j + 1}t^{j - i}; q)_\infty (q^{\mu_i - \lambda_{j+1} + 1} t^{j - i}; q)_\infty}{(q^{\mu_i - \mu_j + 1} t^{j - i}; q)_\infty (q^{\lambda_i - \lambda_{j + 1} + 1} t^{j - i}; q)_\infty (q^{\lambda_i - \mu_j} t^{j - i + 1}; q)_\infty (q^{\mu_i - \lambda_{j+1}} t^{j - i + 1}; q)_\infty}.
\]
\end{theorem}

\begin{corr} \label{corr:mac-sum}
The Macdonald polynomials admit the summation formula
\[
P_\lambda(x; q, t) = \sum_{\mu^1 \prec \cdots \prec \mu^{n-1} \prec \mu^n = \lambda} \prod_{i = 1}^n \psi_{\mu^i / \mu^{i-1}}(q, t) \prod_{i = 1}^n x_i^{|\mu^i| - |\mu^{i-1}|}.
\]
\end{corr}

\subsection{The quantum group $U_q(\gl_n)$}

For a generic value of $q^{1/2}$, let $U_q(\gl_n)$ be the associative algebra with generators $e_i, f_i$ for $i = 1, \ldots, n-1$ and $q^{\pm \frac{h_i}{2}}$ for $i = 1, \ldots, n$ and relations
\begin{align*}
q^{\frac{h_i}{2}} e_i q^{-\frac{h_i}{2}} &= q^{\frac{1}{2}} e_i, \qquad q^{\frac{h_i}{2}} e_{i-1} q^{-\frac{h_i}{2}} = q^{-\frac{1}{2}} e_{i-1}, \qquad q^{\frac{h_i}{2}} f_{i} q^{-\frac{h_i}{2}} = q^{-\frac{1}{2}} f_i, \qquad q^{\frac{h_i}{2}} f_{i-1} q^{-\frac{h_i}{2}} = q^{\frac{1}{2}} f_{i-1}\\
[q^{\frac{h_i}{2}}, e_j] = [&q^{\frac{h_i}{2}}, f_j] = 0 \text{ for $j \neq i, i - 1$}, \qquad [e_i, f_j] = \delta_{ij} \frac{q^{h_i - h_{i+1}} - q^{h_{i+1} - h_i}}{q - q^{-1}}, \qquad [e_i, e_j] = [f_i, f_j] = 0 \text{ for $|i - j| > 1$}\\
q^{\frac{h_i}{2}} \cdot q^{-\frac{h_i}{2}} &= 1, \qquad e_i^2 e_j - (q + q^{-1}) e_i e_j e_i + e_j e_i^2 = 0, \qquad f_i^2 f_j - (q + q^{-1}) f_i f_j f_i + f_j f_i^2 = 0 \text{ for $|i - j| = 1$}.
\end{align*}
We take the coproduct on $U_q(\gl_n)$ defined by 
\begin{align*}
\Delta(e_i) &= e_i \otimes q^{\frac{h_{i+1} - h_i}{2}} + q^{\frac{h_i - h_{i+1}}{2}} \otimes e_i\\
\Delta(f_i) &= f_i \otimes q^{\frac{h_{i+1} - h_i}{2}} + q^{\frac{h_i - h_{i+1}}{2}} \otimes f_i \\
\Delta(q^{\frac{h_i}{2}}) &= q^{\frac{h_i}{2}} \otimes q^{\frac{h_i}{2}}.
\end{align*}
Denote the subalgebra generated by $f_i$ and $q^{\frac{h_i}{2}}$ by $U_q(\bb_-)$.  For each $r < n$, the subalgebra generated by $e_1, \ldots, e_{r-1}$, $f_1, \ldots, f_{r-1}$, and $q^{\frac{h_1}{2}}, \ldots, q^{\frac{h_r}{2}}$ forms a copy of $U_q(\gl_r)$ within $U_q(\gl_n)$. Finally, we denote the finite dimensional irreducible $U_q(\gl_n)$-representation corresponding to a dominant integral signature $\lambda$ by $L_\lambda$.

\subsection{Etingof-Kirillov~Jr. approach to Macdonald polynomials}

In \cite{EK}, Etingof and Kirillov Jr. gave an interpretation of Macdonald polynomials in terms of representation-valued traces of $U_q(\gl_n)$.  Let $W_{k-1}$ denote the $U_q(\gl_n)$-representation
\[
L_{((k - 1)(n - 1), - (k - 1), \ldots, -(k - 1))} = \Sym^{(k-1)n}(\CC^n) \otimes (\det)^{-(k - 1)},
\]
and choose an isomorphism $W_{k-1}[0] \simeq \CC \cdot w_{k-1}$ for some $w_{k-1} \in W_{k-1} [0]$ which spans the $1$-dimensional zero weight space $W_{k-1}[0]$.  Define the weight $\rho_n = \left(\frac{n-1}{2}, \ldots, \frac{1 - n}{2}\right)$. Writing $\rho$ for $\rho_n$, for a signature $\lambda$, there exists a unique intertwiner 
\[
\Phi_\lambda^n: L_{\lambda + (k - 1)\rho} \to L_{\lambda + (k - 1)\rho} \otimes W_{k-1}
\]
normalized to send the highest weight vector $v_{\lambda + (k-1)\rho}$ in $L_{\lambda + (k-1)\rho}$ to 
\[
v_{\lambda + (k - 1)\rho} \otimes w_{k-1} + (\text{lower order terms}),
\]
where $(\text{lower order terms})$ denotes terms of weight lower than $\lambda + (k - 1)\rho$ in the first tensor coordinate. Traces of these intertwiners lie in $W_{k-1}[0] = \CC \cdot w_{k-1}$ and yield Macdonald polynomials when interpreted as scalar functions via the identification $w_{k-1} \mapsto 1$.  Write $x^h$ for $x^h = x_1^{h_1} \cdots x_n^{h_n}$, where in any $U_q(\gl_n)$-representation we interpret $x_i^{h_i}$ as acting on the $\mu$ weight space by $x_i^{\mu_i}$.

\begin{theorem}[{\cite[Theorem 1]{EK}}] \label{thm:ek}
The Macdonald polynomial $P_\lambda(x; q^2, q^{2k})$ is given by 
\[
P_\lambda(x; q^2, q^{2k}) = \frac{\Tr(\Phi_\lambda^n x^h)}{\Tr(\Phi_0^n x^h)}.
\]
\end{theorem}

\begin{prop}[{\cite[Main Lemma]{EK}}] \label{prop:ek-denom}
On $L_{(k - 1)\rho}$, the trace may be expressed explicitly as
\begin{align*}
\Tr(\Phi_0^nx^h) &= (x_1 \cdots x_n)^{-\frac{(k-1)(n - 1)}{2}} \prod_{s = 1}^{k - 1} \prod_{i < j} (x_i - q^{2s} x_j).
\end{align*}
\end{prop}

\begin{remark}
Our notation for Macdonald polynomials is related to that of \cite{EK} via $P^{EK}_{\lambda}(x; q, t) = P_{\lambda}(x; q^2, t^2)$.
\end{remark}

\subsection{Gelfand-Tsetlin basis}

The representation $L_\lambda$ of $U_q(\gl_n)$ admits a basis $\{v_{\bbmu}\}$ indexed by Gelfand-Tsetlin patterns $\bbmu$ subordinate to $\lambda$. The weight of a basis vector $v_{\bbmu}$ is
\[
\wt(v_{\bbmu}) = \Big(|\mu^n| - |\mu^{n-1}|, \ldots, |\mu^2| - |\mu^1|, |\mu^1|\Big).
\]
It was shown in \cite{UTS} that these basis vectors may be expressed in terms of lowering operators $d_{r, i}$ in $U_q(\gl_r) \cap U_q(\bb_-) \subset U_q(\gl_n)$ applied to the highest weight vector $v_\lambda$.  More precisely, we have the following.
\begin{prop}[{\cite[Theorem 2.9]{UTS}}] \label{prop:gt-lower}
There exist lowering operators $d_{r, i} \in U_q(\gl_r) \cap U_q(\bb_-)$ so that the Gelfand-Tsetlin basis vectors are given by 
\[
v_{\bbmu} = d_1^{\mu^1} d_2^{\mu^2 - \mu^1} \cdots d_n^{\mu^n - \mu^{n-1}} v_\lambda,
\]
where $d_r^{\tau} = d_{r, 1}^{\tau_1} \cdots d_{r, r}^{\tau_r}$ for a partition $\tau$.
\end{prop}

\subsection{Statement of the main results}

Computing the trace of $U_q(\gl_n)$-intertwiners in Theorem \ref{thm:ek} in the Gelfand-Tsetlin basis of $L_{\lambda + (k - 1)\rho}$ yields an expression for $P_\lambda(x; q^2, t^2)$ as a summation over Gelfand-Tsetlin patterns subordinate to $\lambda + (k - 1)\rho$.  Our main result shows that diagonal matrix elements of these intertwiners are given by application of Macdonald's operators to a simple kernel.

{\renewcommand{\thetheorem}{\ref{thm:mat-elt}}
\begin{theorem}
In the Gelfand-Tsetlin basis, the diagonal matrix element of $\Phi_\lambda^n$ on the basis vector corresponding to the Gelfand-Tsetlin pattern
\[
\{\sigma^1 \prec \cdots \prec \sigma^{n-1} \prec \lambda + (k - 1)\rho\}
\]
with $\sigma^l_i = \mu_i + (k - 1)\frac{n + 1 - 2i}{2}$ is given by
\[
c(\mu, \lambda) = \frac{\prod_{a = 1}^{k-1} D_{n-1, q^{2\bmu}}(q^{2a}; q^{-2}, q^{2(k-1)})\prod_{i \leq j} [\lambda_i - \mu_j + k (j - i)]_{k-1} \prod_{i < j}[\mu_i - \lambda_j + k (j - i) + k - 2]_{k-1}}{\prod_{i \leq j} [\mu_i - \mu_j + k(j - i) + k - 1]_{k-1} \prod_{i < j}[\lambda_i - \lambda_j + k (j - i) - 1]_{k - 1}},
\]
where $\bmu_i = \mu_i - k(i-1)$, $[m] = \frac{q^m - q^{-m}}{q - q^{-1}}$, and 
\[
D_{n-1, q^{2\bmu}}(u; q^2, t^2) = \sum_{r = 0}^{n - 1} (-1)^{n-1-r} u^{n-1-r} D_{n - 1, q^{2\bmu}}^r(q^2, t^2).
\]
\end{theorem}
\addtocounter{theorem}{-1}}

Using Theorem \ref{thm:mat-elt}, we give a new representation-theoretic proof of Macdonald's branching rule.

{\renewcommand{\thetheorem}{\ref{thm:branch-spec}}
\begin{theorem}
At $t = q^k$ for positive integer $k$, we have 
\[
P_\lambda(x_1, \ldots, x_n; q^2, q^{2k}) = \sum_{\mu \prec \lambda} x_n^{|\lambda| - |\mu|} P_\mu(x_1, \ldots, x_{n-1}; q^2, q^{2k}) \psi_{\lambda/\mu}(q^2, q^{2k})
\]
with 
\[
\psi_{\lambda/\mu}(q^2, q^{2k}) = \frac{\prod_{i \leq j}[\lambda_i - \mu_j + k(j - i) + k - 1]_{k-1} \prod_{i < j} [\mu_i - \lambda_j + k(j - i) - 1]_{k-1}}{\prod_{i \leq j} [\mu_i - \mu_j + k (j - i) + k - 1]_{k-1} \prod_{i < j} [\lambda_i - \lambda_j + k (j - i) - 1]_{k-1}}.
\]
\end{theorem}
}

\begin{remark}
This formulation is equivalent to that of Theorem \ref{thm:branch}.  To see this, note that for each $\lambda$ and $\mu$ the branching coefficients $\psi_{\lambda/\mu}(q, t)$ are rational functions in $q$ and $t$ and are therefore uniquely determined by their values at $(q^2, q^{2k})$ for all positive integers $k$.
\end{remark}

\begin{remark}
Theorem \ref{thm:ek} gives $P_\lambda(x; q^2, q^{2k})$ as a summation over Gelfand-Tsetlin patterns subordinate to $\lambda + (k - 1)\rho$ and Macdonald's branching rule gives it as a summation over Gelfand-Tsetlin patterns subordinate to $\lambda$.  Our result explains how these summations over different index sets are related.
\end{remark}

\subsection{Maps between spherical DAHA's of different rank}

Denote by $\HH_n(q, t)$ and $e\HH_n(q, t)e$ the double affine Hecke algebra of $GL_n$ and its spherical part (see Section 3 for precise definitions).  An essential ingredient in our proof is a map
\[
\bRes_l(q^2): e\HH_{nl}(q^{-2l}, q^2)e \to e \HH_n(q^{-2}, q^{2l})e
\]
between spherical DAHA's of different ranks which results from the Dunkl-Kasatani conjecture of \cite{Dun, Eno, ES, Kas}. We show in Theorem \ref{thm:daha-commute} and Corollary \ref{corr:commute} that $\bRes_l(q^2)$ commutes with Cherednik's $SL_2(\ZZ)$-action on DAHA and that it intertwines the map $\Res_l(q^2): \CC[(X_i^a)^{\pm 1}]^{S_{nl}} \to \CC[X_i^{\pm 1}]^{S_n}$ of spherical polynomial representations given by
\[
\Res_l(q^2): X_i^a \mapsto q^{2(1 - l + 2a)}X_i.
\]

\begin{remark}
Such maps were considered in the rational limit in \cite{CEE}, \cite[Theorem 7.11]{EGL}, and \cite{Sun}.
\end{remark}
 
\subsection{Degenerations of our results and connections to recent work}

Considering our results under the many degenerations of Macdonald polynomials to other special functions yields some connections to recent literature and some interpretations of independent interest.  In this section, we discuss the Heckman-Opdam, Jack, and Hall-Littlewood limits and a generalization to the Macdonald functions of \cite{ES2}.

\begin{itemize}
\item In the quasi-classical limit $q = e^\eps$, $t = q^k$, $\lambda = \lfloor \eps^{-1} \Lambda \rfloor$, $x = e^{\eps X}$, and $\eps \to 0$, the Macdonald polynomials become the Heckman-Opdam hypergeometric functions introduced in \cite{HO, He, Opd, Opd2}.  These functions were recently realized as integrals over Gelfand-Tsetlin polytopes in \cite{BG} by taking a scaling limit of Corollary \ref{corr:mac-sum}.  In \cite{Sun}, the expression of \cite{BG} was lifted to an integral over dressing orbits of a Poisson-Lie group by integration over the Liouville tori and an adjunction procedure involving Calogero-Sutherland Hamiltonians.  The techniques of this paper degenerate to the techniques used in \cite{Sun} under the degeneration from Macdonald-Ruijsenaars to Calogero-Sutherland integrable systems.

\item The Jack polynomials are a scaling limit of Macdonald polynomials under the specialization $t = q^k$ and the limit $q \to 1$ and have a similar branching rule.  They were given in \cite{Eti2} as traces of intertwiners of $U(\gl_n)$-modules using a degeneration of the Etingof-Kirillov~Jr. construction, under which our methods degenerate to a representation-theoretic proof of the Jack branching rule.

\item In the specialization $q = 0$, the Macdonald polynomials become the Hall-Littlewood polynomials. In \cite{Ven}, a summation expression was given for matrix elements of the $U_q(\gl_n)$-intertwiners $\Phi_\lambda^n$ in the Gelfand-Tsetlin basis; this expression factors and becomes particularly simple in the Hall-Littlewood limit.  In the notation of \cite{Ven}, $\Omega_{\beta/\mu}(q^2, q^{2k})$ can be non-zero only if $\mu_i \leq \beta_i \leq \mu_i + (k-1)$, meaning that the prelimit expression of \cite[Theorem 1.3]{Ven} is a sum over an index set similar to that which appears in Proposition \ref{prop:diag-coeff}.  It would be interesting to understand if the factorization which results from degenerating \cite[Theorem 1.3]{Ven} may be obtained by degenerating our Theorem \ref{thm:mat-elt}.

\item Replacing the finite dimensional module $L_{\lambda + (k - 1)\rho}$ by the Verma module $M_\lambda$ in the Etingof-Kirillov~Jr. construction yields the Macdonald functions of \cite{ES2}.  In particular, for a (possibly non-integral) $\lambda$, for the normalization factor
\[
\chi_{k-1}(\lambda) = \prod_{a = 1}^{k-1} \prod_{i < j} (1 - q^{-2(\lambda_i - \lambda_j + j - i) + 2a})
\]
and $\wPsi_\lambda: M_\lambda \to M_\lambda \otimes W_{k - 1}$ the unique intertwiner so that 
\[
\wPsi_\lambda(v_\lambda) = v_\lambda \otimes \chi_{k-1}(\lambda) w_{k-1} + (\text{lower order terms}),
\]
the Macdonald function is the joint eigenfunction of the $D_{n, e^x}^r(q^2, q^{2k})$ given by
\[
\psi(\lambda, x) = \frac{q^{2(k - 1)(\rho, \lambda -\rho)} \Tr(\wPsi_{\lambda - \rho} x^h)}{\prod_{i < j} \prod_{a = 1}^{k - 1} (q^a e^{(x_i - x_j)/2} - q^{-a} e^{(x_j - x_i)/2})}.
\]
Note that our notation is related to that of \cite{ES2} by the substitution $k \mapsto k + 1$. For a dominant integral signature $\lambda$ and $\tau$ a dominant weight in the root lattice, if $\lambda_i - \lambda_{i+1} \geq l$ for some $l \gg 0$, the quotient map $M_{\lambda} \to L_\lambda$ is an isomorphism in the $(\lambda - \tau)$-weight space.  The fact from \cite{ES2} that the branching coefficient for Macdonald functions is a rational function in $q^{\lambda_i}$ and $q^{\mu_i}$ therefore implies the branching rule
\[
\psi(\lambda, x) = \sum_{\mu_1 \in \lambda_1 - \ZZ_{\geq 0}} \cdots \sum_{\mu_{n-1} \in \lambda_{n-1} - \ZZ_{\geq 0}} \wpsi_{\lambda/\mu}^k(q^2) \psi(\mu, x_1, \ldots, x_{n-1}) e^{x_n (|\lambda| - |\mu|)}
\]
with branching coefficient given by
\begin{multline*}
\wpsi_{\lambda / \mu}^k(q^2) = q^{-(n-1)k(k-1)/2} q^{2(k - 1)((\rho, \lambda - \rho) - (\rho, \mu - \rho))} \frac{\chi_{k-1}(\lambda)}{\chi_{k-1}(\mu)} \psi_{\lambda - k \rho / \mu - k \rho}(q^2, q^{2k}) \\
=  (q - q^{-1})^{(k-1)(n - 1)} \frac{\prod_{i < j} [\lambda_i - \lambda_j + j - i - 1]_{k-1} \prod_{i \leq j}[\lambda_i - \mu_j + k - 1]_{k - 1} \prod_{i < j} [\mu_i - \lambda_j - 1]_{k - 1}}{\prod_{i < j} [\mu_i - \mu_j + j - i - 1]_{k-1}\prod_{i \leq j} [\mu_i - \mu_j + k - 1]_{k - 1} \prod_{i < j} [\lambda_i - \lambda_j - 1]_{k - 1}}.
\end{multline*}
Our techniques apply to this setting.  For any $M > 0$, there is some $l > 0$ so that if $|\tau| < M$ and $\lambda_i - \lambda_{i+1} \geq l$, the matrix elements of $\wPsi_{\lambda - \rho}$ on the Gelfand-Tsetlin basis elements of $M_{\lambda - \rho}$ of weight $\lambda - \rho - \tau$ coincide with those of $\chi_{k-1}(\lambda) \Phi_{\lambda - k\rho}^n$.  As shown in \cite{ES2}, the matrix elements are rational functions, hence coincide with the expression of Theorem \ref{thm:mat-elt} for all (possibly non-integral) $\lambda$.  Applying a adjunction argument similar to that of the polynomial case yields the branching rule for Macdonald functions.
\end{itemize}

\subsection{Outline of method and organization}

We briefly outline our method.  Our main technical result is Theorem \ref{thm:daha-commute}, which constructs and characterizes a map $\bRes_l(q^2)$ between spherical DAHA's of rank $nl$ and $n$.  We use Theorem \ref{thm:daha-commute} to relate Macdonald difference operators in $n$ variables at $t = q^l$ to Macdonald difference operators in $nl$ variables at $t = q^{1/l}$.  Combining this with an explicit summation expression for $U_q(\gl_n)$ matrix elements given in \cite{AS}, we obtain in Theorem \ref{thm:mat-elt} a new expression for diagonal $U_q(\gl_n)$ matrix elements as the application of Macdonald difference operators to an explicit kernel.

To obtain Macdonald's branching rule, we interpret the Etingof-Kirillov~Jr. expression for the Macdonald polynomial $P_\lambda(x; q^2, q^{2k})$ as a summation formula over Gelfand-Tsetlin patterns subordinate to $\lambda + (k - 1)\rho$.  Applying Theorem \ref{thm:mat-elt}, the symmetry identity, and summation by parts reduces this to the summation over Gelfand-Tsetlin patterns subordinate to $\lambda$ found in the branching rule. 

The remainder of this paper is organized as follows.  In Section 2, we give some necessary background on Macdonald polynomials and reformulate the results in a convenient form.  In Section 3, we define a map $\bRes_l(q^2)$ between spherical double affine Hecke algebras of different rank and prove the key Theorem \ref{thm:daha-commute} which allows us to compute the image of a certain Macdonald operator in Lemma \ref{lem:res-diff}.  In Section 4, we prove the main Theorem \ref{thm:mat-elt} on matrix elements of $U_q(\gl_n)$-intertwiners by applying the technique developed in Section 3 and a formula from \cite{AS}.  In Section 5, we put everything together to derive a new proof of Macdonald's branching rule.  Section 6 contains some technical manipulations of the result of \cite{AS} postponed from Section 4.

\subsection{Acknowledgments}

The author thanks P. Etingof for helpful discussions. Y.~S. was supported by a NSF graduate research fellowship (NSF Grant \#1122374).

\section{Quantum groups and Macdonald polynomials}

\subsection{Notations}

We will frequently need to consider expressions involving a signature and various shifts; we collect here the conventions we use to denote these.  Set $\rho_{n, i} = \frac{n + 1}{2} - i$ and $\one = (1, \ldots, 1)$.  For any set of indices $I$, let $\one_I$ denote the vector with $1$'s in those indices and $0$'s elsewhere.  Define $\wrho_n = \rho_n - \frac{n - 1}{2}\one$ so that $\wrho_{n, i} = -(i - 1)$ and $\wrho_{n - 1, i} = \wrho_{n, i}$.  For any signature $\lambda$, define the shifts $\wlambda = \lambda + (k - 1)\wrho$ and $\blambda = \lambda + k \wrho$ so that $\wlambda_i = \lambda_i - (k - 1)(i - 1)$ and $\blambda_i = \lambda_i - k (i - 1)$.  Finally, denote by $[a] = \frac{q^a - q^{-a}}{q - q^{-1}}$ the $q$-number, $[a]! = [a] \cdot [a - 1] \cdots [1]$ the $q$-factorial, and $[a]_m = [a]\cdot [a - 1] \cdots [a - m + 1]$ the falling $q$-factorial.

\subsection{Macdonald symmetry identity}

In this subsection, we state the Macdonald symmetry identity and use it to produce conjugates of the Macdonald difference operators acting diagonally the Macdonald polynomials via their index.

\begin{prop}[Macdonald symmetry identity] \label{prop:mac-sym}
We have 
\[
P_\lambda(q^{2\mu + 2k \rho}; q^2, q^{2k}) = \prod_{i < j} \frac{[\lambda_i - \lambda_j + k (j - i) + k-1]_k}{[\mu_i - \mu_j + k (j - i) + k - 1]_k} P_\mu(q^{2\lambda + 2k \rho}; q^2, q^{2k}).
\]
\end{prop}

We would like now to produce Macdonald operators acting on indices of Macdonald polynomials. For this, we abuse notation to write $D_{n - 1, q^{2\bmu}}^r$ for difference operators acting on additive indices $\bmu$ as well as multiplicative variables $q^{2\bmu}$.

\begin{prop} \label{prop:sym-diag}
The operator
\[
\wtilde{D}_{n - 1, q^{2\bmu}}^r(q^2, q^{2k}) = \prod_{i < j} [\bmu_i - \bmu_j + k - 1]_k \circ D_{n - 1, q^{2\bmu}}^r(q^2, q^{2k}) \circ \prod_{i < j} [\bmu_i - \bmu_j + k - 1]_k^{-1}
\]
satisfies
\[
\wtilde{D}_{n - 1, q^{2\bmu}}^r(q^2, q^{2k}) = \sum_{|I| = r} \prod_{i \in I, j \notin I, i > j} \frac{[\bmu_i - \bmu_j + k][\bmu_i - \bmu_j - k + 1]}{[\bmu_i - \bmu_j][\bmu_i - \bmu_j + 1]} T_{q^2, I}
\]
and
\[
\wtilde{D}_{n - 1, q^{2\bmu}}^r(q^2, q^{2k}) P_\mu(x; q^2, q^{2k}) = e_r(x) P_\mu(x; q^2, q^{2k}).
\]
\end{prop}
\begin{proof}
The expression for $\wtilde{D}_{n - 1, q^{2\bmu}}^r(q^2, q^{2k})$ follows by direct computation, and the eigenvalue identity from the Macdonald symmetry identity. 
\end{proof}

\subsection{Adjoints of Macdonald difference operators}

We would like now to consider adjoints of Macdonald operators with respect to a Jackson-type inner product.  Fix lower and upper limits $\zeta = (\zeta^-, \zeta^+)$ with $\zeta^- = (\zeta_1^-, \ldots, \zeta_{n-1}^-)$, $\zeta^+ = (\zeta_1^+, \ldots, \zeta_{n-1}^+)$, and $\zeta_i^+ - \zeta_i^{-} \in \ZZ_{\geq 0}$.  Define the inner product 
\[
\langle f, g \rangle_{\zeta} := \sum_{\mu = \zeta^-}^{\zeta^+} f(q^{2\mu}) g(q^{2\mu}),
\]
where we define the iterated summation symbol by 
\begin{equation} \label{eq:summation}
\sum_{\mu = \zeta^-}^{\zeta^+} := \sum_{\mu_1 = \zeta^-_1}^{\zeta^+_1} \cdots \sum_{\mu_{n-1} = \zeta^-_{n-1}}^{\zeta^+_{n-1}}.
\end{equation}
We will consider situations where $g$ vanishes along a border of the region of summation.  In particular, we say that the function $g(q^{2\mu})$ is $(\zeta, l)$-adapted if $g(q^{2\mu}) = 0$ on the set 
\[
\{\mu \mid \text{$\zeta^+_i < \mu_i \leq \zeta^+_i + l$ or $\zeta^-_i - l \leq \mu_i < \zeta^-_i$ for any $i$}\}.
\]
We now characterize adjoints with respect to $\langle, \rangle_\zeta$ when applied to an $(\zeta, l)$-adapted function.

\begin{prop} \label{prop:mac-adj}
If $f(q^{2\mu})$ is $(\zeta, l)$-adapted, we have for any $g$ that 
\[
\left\langle \prod_{i = 1}^l \wD_{n - 1, q^{2\bmu}}^{r_{l + 1 - i}}(q^2, q^{2k})^\dagger f, g\right \rangle_{(\zeta^-, \zeta^+ + l\one)} = \left\langle f, \prod_{i = 1}^l \wD_{n - 1, q^{2\bmu}}^{r_i}(q^{2}, q^{2k}) g \right\rangle_\zeta,
\]
where 
\begin{align*}
\wtilde{D}_{n - 1, q^{2\bmu}}^r(q^2, q^{2k})^\dagger = \prod_{i < j} [\bmu_i - \bmu_j + k - 1]_{k-1}^{-1} \circ D_{n - 1, q^{2\bmu}}^r(q^{-2}, q^{2(k - 1)})  \circ \prod_{i < j} [\bmu_i - \bmu_j + k - 1]_{k-1}.
\end{align*}
\end{prop}
\begin{proof}
First, we check by a direct computation that 
\[
\wtilde{D}_{n - 1 q^{2\bmu}}^r(q^2, q^{2k})^\dagger = \sum_{|I| = r} \prod_{i \in I, j \notin I, i > j} \frac{[\bmu_i - \bmu_j + k - 1][\bmu_i - \bmu_j - k]}{[\bmu_i - \bmu_j - 1][\bmu_i - \bmu_j]} T_{q^{-2}, I}.
\]
Now, for any subset of indices $I$, we have
\begin{equation} \label{eq:t-swap}
\langle f, T_{q^2, I} g \rangle_\zeta = \sum_{\mu = \zeta^-}^{\zeta^+} f(q^{2\mu}) g(q^{2(\mu + \one_I)}) = \sum_{\mu = \zeta^- + \one_I}^{\zeta^+ + \one_I} f(q^{2(\mu - \one_I)}) g(q^\mu) = \langle T_{q^{-2}, I} f, g \rangle_{(\zeta^- + \one_I, \zeta^+ + \one_I)}.
\end{equation}
Using this, we induct on $l$.  For $l = 1$, we have
\begin{align*}
\left\langle \wD_{n - 1, q^{2\bmu}}^{r}(q^{2}, q^{2k})^\dagger f, g \right\rangle_{(\zeta^-, \zeta^+ + \one)} &= \sum_{\mu = \zeta^-}^{\zeta^+ + \one} \sum_{|I| = r} \prod_{i \in I, j \notin I, i > j} \frac{[\bmu_i - \bmu_j + k - 1][\bmu_i - \bmu_j - k]}{[\bmu_i - \bmu_j - 1][\bmu_i - \bmu_j]}f(q^{2(\mu - \one_I)})  g(q^{2\mu})\\
&= \sum_{|I| = r} \sum_{\mu = \zeta^- - \one_I}^{\zeta^+} f(q^{2\mu}) T_{q^2, I}\left(\prod_{i \in I, j \notin I, i > j} \frac{[\bmu_i - \bmu_j + k - 1][\bmu_i - \bmu_j - k]}{[\bmu_i - \bmu_j - 1][\bmu_i - \bmu_j]}\right) g(q^{2(\mu + \one_I)})\\
&= \sum_{|I| = r} \sum_{\mu = \zeta^-}^{\zeta^+} f(q^{2\mu}) \prod_{i \in I, j \notin I, i > j} \frac{[\bmu_i - \bmu_j + k][\bmu_i - \bmu_j - k + 1]}{[\bmu_i - \bmu_j][\bmu_i - \bmu_j + 1]} g(q^{2(\mu + \one_I)})\\
&= \left\langle f, \wD_{n - 1, q^{2\bmu}}^{r}(q^{2}, q^{2k}) g \right\rangle_\zeta,
\end{align*}
where the second equality follows from (\ref{eq:t-swap}), the third follows because $f$ is $(\zeta, 1)$-adapted, and the last equality follows from Proposition \ref{prop:sym-diag}.  Now suppose the claim holds for $l - 1$.  If $f$ is $(\zeta, l)$-adapted, then $\wD_{n - 1, q^{2\bmu}}^{r_1}(q^{2}, q^{2k})^\dagger f$ is $(\zeta^-, \zeta^+ + \one, l - 1)$-adapted, so we have by applying the cases of $l - 1$ and then $1$ that
\begin{align*}
\left\langle \prod_{i = 1}^l \wD_{n - 1, q^{2\bmu}}^{r_{l + 1 - i}}(q^2, q^{2k})^\dagger f, g\right \rangle_{(\zeta^-, \zeta^+ + l\one)} &= \left\langle \wD_{n - 1, q^{2\bmu}}^{r_{1}}(q^2, q^{2k})^\dagger f, \prod_{i = 1}^{l - 1} \wD_{n - 1, q^{2\bmu}}^{r_{l + 1 - i}}(q^2, q^{2k}) g\right \rangle_{(\zeta^-, \zeta^+ + \one)}\\
&= \left\langle f, \prod_{i = 1}^l \wD_{n - 1, q^{2\bmu}}^{r_i}(q^{2}, q^{2k}) g \right\rangle_\zeta. \qedhere
\end{align*}
\end{proof}

\subsection{Reformulating the Etingof-Kirillov~Jr. construction}

In this subsection we shift the weights of the representations used in the Etingof-Kirillov~Jr. construction to make restriction from $U_q(\gl_n)$ to $U_q(\gl_{n-1})$ more notationally convenient. For a partition $\lambda$, define the intertwiner
\[
\wPhi_\lambda^n: L_{\lambda + (k - 1) \wrho} \to L_{\lambda + (k - 1) \wrho} \otimes W_{k-1}
\]
to be $\wPhi_\lambda^n = \Phi_{\lambda}^n \otimes \id_{(\det)^{-\frac{(k-1)(n-1)}{2}}}$.  We now rephrase Theorem \ref{thm:ek} in terms of the intertwiners $\wPhi_\lambda^n$.

\begin{corr} \label{corr:ek-shift}
The Macdonald polynomial $P_\lambda(x; q^2, q^{2k})$ is given by
\[
P_\lambda(x; q^2, q^{2k}) = \frac{\Tr(\wPhi_{\lambda}^n x^h)}{\Tr(\wPhi_0^n x^h)}.
\]
\end{corr}
\begin{proof}
This follows from Theorem \ref{thm:ek} and the relation
\[
\Tr(\wPhi_{\lambda}^n x^h) = \Tr(\Phi_\lambda^n x^h) (x_1 \cdots x_n)^{-\frac{(k-1)(n-1)}{2}}. \qedhere
\]
\end{proof}

\begin{corr} \label{corr:denom}
The denominator in Corollary \ref{corr:ek-shift} is given by 
\[
\Tr(\wPhi_0^n x^h) = (x_1 \cdots x_n)^{-(k-1)(n-1)} \prod_{s = 1}^{k-1} \prod_{i < j} (x_i - q^{2s} x_j).
\]
\end{corr}
\begin{proof}
This follows from Proposition \ref{prop:ek-denom} and the definition of $\wPhi_0^n$.
\end{proof}

\section{Spherical subalgebras of double affine Hecke algebras of different ranks}

\subsection{Double affine Hecke algebras}

Let $\HH_n(q, t)$ denote the double affine Hecke algebra (DAHA) of $GL_n$ defined by \cite{Che2}.  Following \cite{SV}, it is defined as the associative algebra generated by invertible elements $X_1^{\pm 1}, \ldots, X_n^{\pm 1}, Y_1^{\pm 1}, \ldots, Y_n^{\pm 1}$, and $T_1^{\pm 1}, \ldots, T_{n-1}^{\pm 1}$ subject to the relations
\begin{itemize}
\item $(T_i - t^{1/2})(T_i + t^{-1/2}) = 0$, $T_i T_{i+1} T_i = T_{i+1}T_i T_{i+1}$, $[T_i, T_j] = 0$ for $|i - j| \neq 1$;
\item $T_iX_iT_i = X_{i+1}$, $T_i^{-1}Y_i T_i^{-1} = Y_{i+1}$, and $[T_i, X_j] = [T_i, Y_j] = 0$ for $|i - j| > 1$;
\item $[X_i, X_j] = 0$, $[Y_i, Y_j] = 0$, $Y_1 X_1 \cdots X_n = qX_1 \cdots X_nY_1$, and $X_1^{-1} Y_2 = Y_2 X_1^{-1}T_1^{-2}$.
\end{itemize}
Note that $\{T_i\}$ generate a copy of the finite-type Hecke algebra, and $\{T_i, X_j\}$ and $\{T_i, Y_j\}$ generate copies of the affine Hecke algebra. For $\sigma = s_{i_1} \cdots s_{i_l}$ a reduced decomposition in $S_n$, let $T_\sigma = T_{i_1} \cdots T_{i_l}$.  Define the idempotent 
\[
e = \frac{(1 - t)^n}{(t; t)_n} \sum_{\sigma \in S_n} t^{\ell(\sigma)/2} T_\sigma.
\]
The spherical DAHA is defined to be the subalgebra $e\HH_n(q, t) e$.  From the results of \cite{AFS, SV, BS, SV2} surveyed in \cite{Neg} on maps between the Drinfeld double of the elliptic Hall algebra and the spherical DAHA, we may extract the following small set of generators.

\begin{lemma}[{\cite[Section 2.4]{Neg}}] \label{lem:daha-pgen}
The elements $ep_1(Y)e$, $ep_{-1}(Y)e$, $ep_1(X)e$, and $e p_{-1}(X)e$ generate $e\HH_{n}(q, t)e$.
\end{lemma}

\subsection{Polynomial representation of DAHA and Macdonald operators}

The double affine Hecke algebra admits a faithful polynomial representation $\rho$ on $\CC[X_1^{\pm 1}, \ldots, X_n^{\pm 1}]$ given by 
\begin{align*}
\rho(X_i) &= X_i\\
\rho(T_i) &= t^{1/2} s_i + \frac{t^{1/2} - t^{-1/2}}{X_i / X_{i+1} - 1} (s_i - 1)\\
\rho(Y_i) &= \rho(T_i) \cdots \rho(T_{n-1})s_{n - 1} \cdots s_1 T_{q, X_1} \rho(T_1^{-1}) \cdots \rho(T_{i - 1}^{-1}),
\end{align*}
where $s_i$ exchanges $X_i$ and $X_{i+1}$ and $T_{q, X_1}$ is the $q$-shift operator in $X_1$.  The action of elements of $e \HH_n(q, t) e$ on the symmetric part of the polynomial representation yields the Macdonald operators.

\begin{prop} \label{prop:che-mac}
When restricted to $\CC[X_1^{\pm 1}, \ldots, X_n^{\pm 1}]^{S_n}$, the action of $e \cdot e_r(Y_1, \ldots, Y_n) \cdot e$ is given by Macdonald's operator
\[
\rho(e \cdot e_r(Y_1, \ldots, Y_n) \cdot e) = \rho(e_r(Y_1, \ldots, Y_n)) = D_{n, X}^r(q, t).
\]
In particular, for any $n$-variable symmetric polynomial $f$, the operator 
\[
L_f = f(Y_1, \ldots, Y_n)
\]
is diagonalized on $P_\lambda(X; q, t)$ with eigenvalue $f(q^{\lambda}t^{\rho})$.
\end{prop}

\begin{prop} \label{prop:mac-negative}
When restricted to $\CC[X_1^{\pm 1}, \ldots, X_n^{\pm 1}]^{S_n}$, the action of $e \cdot p_1(Y^{-1}) \cdot e$ is given by 
\[
D_{n, X}^{n - 1}(q, t) \circ D_{n, X}^n(q, t)^{-1} = t^{-\frac{n - 1}{2}} \sum_{i = 1}^n \prod_{j \neq i} \frac{t x_j - x_i}{x_j - x_i} T_{q^{-1}, i}.
\]
\end{prop}
\begin{proof}
This follows from Proposition \ref{prop:che-mac} and the fact that $e p_1(Y^{-1}) e = e \cdot e_{n - 1}(Y) e_n(Y^{-1}) \cdot e$.
\end{proof}

\begin{remark}
By faithfulness, we will refer interchangeably to elements of the DAHA and spherical DAHA and their images under the polynomial representation in what follows.
\end{remark}

\subsection{$SL_2(\ZZ)$-action on DAHA}

Define the isomorphisms $\eps(q, t): \HH_n(q, t) \to \HH_n(q^{-1}, t^{-1})$ given by 
\[
\eps(q, t): X_i \mapsto Y_i, Y_i \mapsto X_i, T_i \mapsto T_i^{-1}, q \mapsto q^{-1}, t \mapsto t^{-1}
\]
and $\tau_+(q, t): \HH_n(q, t) \to \HH_n(q, t)$ given by
\[
\tau_+: X_i \mapsto X_i, T_i \mapsto T_i, Y_1 \cdots Y_r \mapsto  q^{-r/2} X_1 \cdots X_r Y_1 \cdots Y_r.
\]
Define also the composition $\tau_- = \eps \tau_+ \eps$.
\begin{prop}[{\cite[Section 3.2.2]{Che}}] \label{prop:sl2z}
The map
\[
\left(\begin{matrix} 1 & 0 \\ 1 & 1 \end{matrix}\right)\mapsto \tau_-, \qquad \left(\begin{matrix} 1 & 1 \\ 0 & 1 \end{matrix}\right) \mapsto \tau_+
\]
defines an action of $SL_2(\ZZ)$ on $\HH_n(q, t)$ which preserves $e \HH_n(q, t)e$. 
\end{prop}

The action of $\tau_+$ in the polynomial representation is realized via conjugation by the Gaussian
\[
\gamma_n(q) = q^{\sum_i \wx_i^2/2},
\]
where $q^{\wx_i} = X_i$.  Here, we view $\gamma_n(q)$ as an element in the completion of $\HH_n(q, t)$ by degree of $X$.

\begin{prop}[{\cite[Section 3.7]{Che}}] \label{prop:gauss-conj}
When evaluated in the polynomial representation, the action of $\tau_+$ on $\HH_n(q, t)$ is given by conjugation by $\gamma_n(q)$. That is, we have the equality
\[
\rho(\tau_+(f)) =\gamma_n(q) \rho(f) \gamma_n(q)^{-1}.
\]
\end{prop}

\subsection{Multiwheel condition and the restriction map}

Following the generalization in \cite{Kas} of the original wheel condition of \cite{FJMM, FJMM2}, we say that $(X_1^0, \ldots, X_{n}^{l-1}) \in \CC^{nl}$ satisfies the \textit{multiwheel condition} if the indices may be permuted so that
\[
X_i^a = X_i^0 t^{a - 1} \text{ for $1 \leq i \leq n$ and $0 \leq a \leq l - 1$}.
\]
Define the ideal $I_{nl}(t) \subset \CC[(X_i^a)^{\pm 1}]$ by 
\[
I_{nl}(t) = \{f \mid f(X) = 0 \text{ if $X$ satisfy the multiwheel condition}\}.
\]
In \cite{Kas}, this ideal was characterized as a $\HH_{nl}(q, t)$-submodule.

\begin{prop}[{\cite[Theorem 6.3]{Kas} and \cite[Theorem 5.10]{ES}}] \label{prop:dk}
The subspace $I_{nl}(t) \subset \CC[(X_i^a)^{\pm 1}]$ is a $\HH_{nl}(q, t)$-submodule and $\CC[(X_i^a)^{\pm 1}] / I_{nl}(t)$ is irreducible.
\end{prop}

\begin{remark}
Along with some finer statements about the structure of $I_{nl}(t)$ and other submodules defined by similar multiwheel conditions, Proposition \ref{prop:dk} was conjectured in \cite[Conjecture 6.4]{Kas} and in the rational limit in \cite{Dun}.  These statements are known as the Dunkl-Kasatani conjecture and were later proven in \cite{Eno} for generic values of parameters and for all values of parameters in \cite[Theorem 5.10]{ES}.
\end{remark}

Define the map $\Res_l(q^2): \CC[(X_i^a)^{\pm 1}]^{S_{nl}} \to \CC[X_i^{\pm 1}]^{S_n}$ by
\[
\Res_{l}(q^2)(X_i^a) = q^{1 - l + 2a} X_i.
\]
The kernel of $\Res_l(q^2)$ is $I_{nl}^{S_{nl}}(q^2)$, so $\Res_l(q^2)$ induces by Proposition \ref{prop:dk} an action of $e\HH_{nl}(q^{-2l}, q^2)e$ on $\CC[X_i^{\pm 1}]^{S_n}$, giving a map
\[
\wRes_l(q^2): e \HH_{nl}(q^{-2l}, q^2) e \to \End(\CC[X_i^{\pm 1}]^{S_n}).
\]
We claim that this map factors through the polynomial representation
\[
e \HH_{n}(q^{-2}, q^{2l}) e \to \End(\CC[X_i^{\pm 1}]^{S_n})
\]
via a map of algebras $\bRes_l(q^{2}): e \HH_{nl}(q^{-2l}, q^2) e \to e \HH_{n}(q^{-2}, q^{2l}) e$. 

\begin{theorem} \label{thm:daha-commute}
The map $\bRes_l(q^2): e \HH_{nl}(q^{-2l}, q^2) e \to e \HH_{n}(q^{-2}, q^{2l}) e$ defined by
\begin{align*}
\bRes_l(q^2)(e p(X_i^a) e) &= e p(q^{1 - l}X_1, \ldots, q^{l - 1} X_1, \ldots, q^{1 - l} X_n, \ldots, q^{l - 1} X_n) e \text{ and}\\
\bRes_l(q^2)(e p(Y_i^a) e) &= e p(q^{1 - l}Y_1, \ldots, q^{l - 1} Y_1, \ldots, q^{1 - l} Y_n, \ldots, q^{l - 1} Y_n) e
\end{align*}
for $p \in \CC[(X_i^a)^{\pm 1}]^{S_{nl}}$ is well defined and satisfies

\begin{itemize}
\item[(a)] for any $h \in e\HH_{nl}(q^{-2l}, q^2)e$, as operators on $\CC[(X_i^a)^{\pm 1}]^{S_{nl}}$ we have
\[
\Res_l(q^2) \circ h = \bRes_l(h) \circ \Res_l(q^2);
\]

\item[(b)] as operators on $e\HH_{nl}(q^{-2l}, q^2)e$, we have
\[
\bRes_l(q^{-2}) \circ \eps_{nl}(q^{-2l}, q^2) = \eps_n(q^{-2}, q^{2l}) \circ \bRes_l(q^{2});
\]

\item[(c)] as operators on $e\HH_{nl}(q^{-2l}, q^2)e$, we have
\[
\bRes_l(q^{2}) \circ \tau_+ = \tau_+ \circ \bRes_l(q^{2}).
\]
\end{itemize}
\end{theorem}
\begin{proof}
We first check that (a) holds on the generating set of Lemma \ref{lem:daha-pgen}.  This is evident for $h \in e\CC[(X_i^a)^{\pm 1}]^{S_{nl}}e$.  For $h = ep_1(Y_i^a)e$, we compute
\begin{align*}
&\Res_l(q^2) D_{nl, X}^1(q^{-2l}, q^2) f(X_1, \ldots, X_n)\\ &= q^{1 - nl} \sum_{i = 1}^n \sum_{a = 0}^{l-1} \prod_{(j, b) \neq (i, a)}\frac{q^{2+2a} X_i - q^{2b}X_j}{q^{2a}X_i - q^{2b}X_j} f(q^{1 - l} X_1, \ldots, q^{l - 1}X_1, \ldots, q^{1 - l + 2a - 2l}X_i, \ldots, q^{l-1}X_n) \\
&= q^{1 - nl} \sum_{i = 1}^n \prod_{(j, b) \neq (i, l-1)}\frac{q^{2l} X_i - q^{2b}X_j}{q^{2l - 2}X_i - q^{2b}X_j}f(q^{1 - l}X_1, \ldots, q^{l - 1}X_1, \ldots, q^{1 - l}X_i, \ldots, q^{l-3}X_i, q^{-1-l} X_i, \ldots, q^{l - 1}X_n) \\
&= q^{1 - nl} \sum_{i = 1}^n \prod_{b = 0}^{l-2} \frac{q^2(q^{2l - 2} - q^{2b-2})}{q^{2l-2} - q^{2b}}  \prod_{j \neq i} \prod_{b = 0}^{l-1} \frac{q^2(q^{2l-2}X_i - q^{2b-2}X_j)}{q^{2l-2}X_i - q^{2b} X_j} T_{q^{-2}, X_i} \Res_l(q^2)(f)(X_1, \ldots, X_n)\\
&= q^{1 - nl} \frac{1 - q^{2l}}{1 - q^{2}} \sum_{i = 1}^n \prod_{j \neq i} \frac{q^{2l} X_i - X_j}{X_i - X_j} T_{q^{-2}, X_i} \Res_l(q^2) f(X_1, \ldots, X_n) \\
&= [l]\, D_{n, X}^1(q^{-2}, q^{2l}) \Res_l(q^2) f(X_1, \ldots, X_n),
\end{align*}
which shows that (a) holds for $h = e p_1(Y_i^a)e$.  A similar computation using the expression of Proposition \ref{prop:mac-negative} yields (a) for $h = e p_1((Y_i^a)^{-1}) e$.  We conclude that (a) holds for $h$ in the generating set $ep_1(X)e$, $ep_1(X^{-1})e$, $ep_1(Y)e$, $ep_1(Y^{-1})e$.  Therefore, the stated values of $\bRes_l(q^2)$ extend to a well-defined map satisfying (a).

We now use (a) and the value of $\bRes_l(q^2)$ on the generators to prove (b) and (c).  For (b), by Lemma \ref{lem:daha-pgen}, it suffices to check on $ep_1(Y^{\pm 1})e$ and $ep_1(X^{\pm 1})e$.  We give the computations for $ep_1(X)e$ and $ep_1(Y)e$; the checks for $ep_1(X^{-1})e$ and $ep_1(Y^{-1})e$ are analogous. For the first check, note that
\begin{align*}
\bRes_l(q^{-2}) (\eps_{nl}(q^{-2l}, q^2)(D_{nl, X}^1(q^{-2l}, q^2))) &= \bRes_l(q^{-2}) (p_{1, nl}(X)) = [l]\, p_{1, n}(X), \text{ and}\\
\eps_n(q^{-2}; q^{2l}) (\bRes_l(q^2)(D_{nl, X}^1(q^{-2l}, q^2))) &= \eps_n(q^{-2}; q^{2l}) ([l]\, D_{n, X}^1(q^{-2}, q^{2l})) = [l]\, p_{1, n}(X).
\end{align*}
For the second check, note that 
\begin{align*}
\bRes_l(q^{-2}) (\eps_{nl}(q^{-2l}, q^2)(p_1(X_i^a))) &= \bRes_l(q^{-2})(D_{nl, X}^1(q^{2l}, q^{-2})) = [l]\, D_{n, X}^1(q^2, q^{-2l})\\
\eps_n(q^{-2}; q^{2l}) (\bRes_l(q^2)(p_1(X_i^a))) &= \eps_n(q^{-2}; q^{2l})([l] p_1(X_i)) = [l]\, D_{n, X}^1(q^2, q^{-2l}),
\end{align*}
where we apply the fact from (a) that
\[
\bRes_l(q^2) D_{nl, X}^1(q^{-2l}, q^2) = [l]\, D_{n, X}^1(q^{-2}, q^{2l})
\]
with $q$ and $q^{-1}$ interchanged. This completes the proof of (b).

For (c), note that in $\HH_n(q^{-2l}; q^2)$, we have
\[
\bRes_l(q^2)(q^{-2l\wx_i^a}) = \bRes_l(q^2)(X_i^a) = X_i q^{2(1 - l + 2a)} = q^{-2\wx_i + 2(1 - l + 2a)}.
\]
This implies that
\begin{align*}
\bRes_l(q^2)(\gamma_{nl}(q^{-2l})) = \bRes_l(q^2)(q^{-l \sum_{i, a} (\wx_i^a)^2}) =  q^{-2 \sum_i \wx_i^2 - \frac{1}{l} \sum_{i, a} (1 - l + 2a)^2} = q^{- \frac{1}{l} \sum_{i, a} (1 - l + 2a)^2} \gamma_n(q^{-2}),
\end{align*}
which yields the desired by Proposition \ref{prop:gauss-conj}.

Finally, to obtain the claimed values on $ep(Y)e$ for all $p$, we note by (b) that
\begin{align*}
\bRes_l(ep(Y_i^a)e) &= \bRes_l(q^2)(\eps_{nl}(q^{2l}, q^{-2})(ep(X_i^a)e))\\
&= \eps_n(q^2, q^{-2l})(\bRes_l(q^{-2})(e p(X_i^a)e)) \\
&= \eps_n(q^2, q^{-2l})(ep(q^{l - 1}X_1, \ldots, q^{1 - l}X_1, \ldots, q^{l - 1}X_n, \ldots, q^{1 - l}X_n)e)\\ &= e p(q^{1 - l}Y_1, \ldots, q^{l - 1}Y_1, \ldots, q^{1 - l}Y_n, \ldots, q^{l - 1}Y_n)e. \qedhere
\end{align*}
\end{proof}

\begin{corr} \label{corr:commute}
The map $\bRes_l(q^2)$ commutes with the action of $SL_2(\ZZ)$ on the spherical DAHA.
\end{corr}
\begin{proof}
By Theorem \ref{thm:daha-commute}(bc) and the fact that the $SL_2(\ZZ)$-action is implemented via $\eps_+$ and $\tau_+$.
\end{proof}

\subsection{Extending the restriction map} \label{sec:res-extend}

In our application, we must extend the restriction map slightly.  The assignment
\[
\Res_l(q^2)((X_i^a)^{1/2}) = q^{a - (l - 1)/2} X_i^{1/2}
\]
extends $\Res_l(q^2)$ to an operator $\CC[(X_i^a)^{\pm 1/2}]^{S_{nl}} \to \CC[X_i^{1/2}]^{S_n}$.  If we identify elements of the spherical DAHA with difference operators, they define valid operators on the subspace
\[
\prod_{i, a} (X_i^a)^{1/2} \cdot \CC[(X_i^a)^{\pm 1}]^{S_{nl}} \subset \CC[(X_i^a)^{\pm 1/2}]^{S_{nl}},
\]
though they do not in general satisfy the spherical DAHA relations.  We see that Theorem \ref{thm:daha-commute}(a) continues to hold in this setting.

\begin{corr} \label{corr:daha-comm-shift}
For any $h \in e \HH_{nl}(q^{-2l}, q^2)e$, as operators on $\prod_{i, a} (X_i^a)^{1/2} \cdot \CC[(X_i^a)^{\pm 1}]^{S_{nl}}$ we have
\[
\Res_l(q^2) \circ h = \bRes_l(h) \circ \Res_l(q^2).
\]
\end{corr}
\begin{proof}
We interpret both sides as operators 
\[
\prod_{i, a} (X_i^a)^{1/2} \cdot \CC[(X_i^a)^{\pm 1}]^{S_{nl}} \to \prod_i X_i^{l/2} \cdot \CC[X_i^{\pm 1}]^{S_n}
\]
and identify $\prod_{i, a} (X_i^a)^{1/2} \cdot \CC[(X_i^a)^{\pm 1}]^{S_{nl}}$ with $\CC[(X_i^a)^{\pm 1}]^{S_{nl}}$ and $\prod_i X_i^{l/2} \cdot \CC[X_i^{\pm 1}]^{S_n}$ with $\CC[X_i^{\pm 1}]^{S_n}$.  For $h \in e \CC[(X_i^a)^{\pm 1}]e$, both sides yield the map of Theorem \ref{thm:daha-commute}.  For $h = e  p(Y_i^a) e$ with $p$ a degree $r$ homogeneous symmetric polynomial, both sides are equal to the map of Theorem \ref{thm:daha-commute} multiplied by $q^{-lr}$.  Together, these give the claim.
\end{proof}

\subsection{Computing $\bRes_l(q^2)$ on a specific operator}

Define the operator
\begin{equation} \label{eq:spec-op}
D_{n, X}(u; q, t) = \sum_{r} (-1)^{n-r} u^{n - r} D_{n, X}^r(q, t).
\end{equation}
Identify $e\HH_{nl}(q^{-2l}, q^2)e$ with its image under the polynomial representation; in this identification, we now compute the image of a specific operator under $\bRes_l$.

\begin{lemma} \label{lem:res-diff}
We have the relation
\[
\bRes_l(q^2)(D_{nl, X}(q^{l+1}; q^{-2l}, q^2)) = \prod_{a = 1}^{l} D_{n, X}(q^{2a}; q^{-2}, q^{2l}).
\]
\end{lemma}
\begin{proof}
Observe that 
\begin{multline*}
D_{nl, X}(u; q^{-2l}, q^2) = \sum_r (-1)^{nl - r} u^{nl - r} D_{nl, X}^r(q^{-2l}, q^2)\\
 = \eps_{nl}(q^{2l}, q^{-2}) \sum_r (-1)^{nl - r} \bar{u}^{nl - r} e_r(X_i^a) = \eps_{nl}(q^{2l}, q^{-2}) \prod_{i, a} (X_i^a - \bar{u}),
\end{multline*}
where $\bar{u} = \eps_{nl}(q^{2l}, q^{-2})(u)$.  Therefore, by Theorem \ref{thm:daha-commute}(b) with $q$ and $q^{-1}$ interchanged we find that 
\begin{align*}
\bRes_l(q^2) D_{nl, X}(u; q^{-2l}, q^2) &= \bRes_l(q^2) \eps_{nl}(q^{2l}, q^{-2})\prod_{i, a} (X_i^a - \bar{u})\\
&= \eps_n(q^2, q^{-2l}) \bRes_l(q^{-2})\prod_{i, a} (X_i^a - \bar{u})\\
&= \eps_n(q^2, q^{-2l}) \prod_{i, a} (q^{-2a + l - 1} X_i - \bar{u})\\
&= \eps_n(q^2, q^{-2l}) \prod_{a = 0}^{l - 1} \left(\sum_r (-1)^{n-r} (q^{2a - l + 1}\bar{u})^{n-r} e_r(X_i)\right)\\
&= \prod_{a = 0}^{l-1} \left(\sum_r (-1)^{n-r} (q^{l - 1 - 2a} u)^{n - r} D_{n, X}^r(q^{-2}, q^{2l}) \right)\\
&= \prod_{a = 0}^{l-1} D_{n, X}(u q^{l - 1 - 2a}; q^{-2}, q^{2l}).
\end{align*}
Setting $u = q^{l + 1}$ implies the desired
\[
\bRes_l(q^2)(D_{nl, X}(q^{l + 1}; q^{-2l}, q^2)) = \prod_{a = 1}^l D_{n, X}(q^{2a}; q^{-2}, q^{2l}). \qedhere
\]
\end{proof}

\section{Computing diagonal matrix elements in the Gelfand-Tsetlin basis}

\subsection{Factorization of matrix elements}

For a choice of $\mu^1, \ldots, \mu^n = \lambda$ so that $\wmu^i \prec \cdots \prec \wmu^n = \wlambda$ forms a Gelfand-Tsetlin pattern subordinate to $\wlambda$, denote the pattern by $\wbbmu$.  Let $c(\wbbmu, \lambda)$ denote the diagonal matrix coefficient of $v_{\wbbmu}$ in $\wPhi_\lambda$. For a signature $\mu \prec \lambda$, let the pattern $\gt(\mu)$ be defined by 
\begin{equation} \label{eq:gt-def}
\gt(\mu)^l_i = \mu_i \text{ for $l < n$}.
\end{equation}
Define $c(\mu, \lambda)$ to be the diagonal matrix coefficient $c(\gt(\wmu), \lambda)$ of $v_{\gt(\wmu)}$ in $\wPhi_\lambda$.

We show that $\wPhi_\lambda$ has non-zero diagonal matrix elements only on basis vectors indexed by patterns of the form $\wbbmu$ and that these elements admit a level-by-level factorization.

\begin{lemma} \label{lem:exclude}
If $v_{\bbmu}$ is not of the form $v_{\wbbmu}$, then $v_{\bbmu}$ has zero diagonal matrix element in $\wPhi_\lambda$.  
\end{lemma}
\begin{proof}
For some $r < n$, we cannot write $\bbmu^r = \wtau$ for any $\tau$.  Let $U \subset W_{k-1}$ be the $U_q(\gl_r)$-submodule consisting of vectors of weight $0$ for $q^{h_{r + 1}}, \ldots, q^{h_n}$ so that $U \simeq L_{(k - 1)(r - 1, -1, \ldots -1)}$ as a $U_q(\gl_r)$-module.  Let $\underline{\bbmu}^r$ denote the truncation of $\bbmu^r$ so that $\underline{\bbmu}^r_i = \bbmu^r_i$.  Consider the Gelfand-Tsetlin pattern $\bbxi$ given by 
\[
\bbxi = \{\gt(\underline{\bbmu}^r) \prec \bbmu^{r+1} \prec \cdots \prec \bbmu^{n-1} \prec \lambda\}.
\]
Let $L_{\bbmu^r} \subset L_{\wlambda}$ be the $U_q(\gl_r)$-submodule with highest weight $\bbmu^r$ generated by $v_{\bbxi}$.   By Proposition \ref{prop:gt-lower}, the diagonal matrix element of $v_{\bbmu}$ lies in $L_{\bbmu^r} \otimes U$, hence is a multiple of the matrix element of $v_{\bbmu}$ in the induced $U_q(\gl_r)$-intertwiner
\[
L_{\bbmu^r} \to L_{\wlambda} \to L_{\wlambda} \otimes W_{k-1} \to L_{\bbmu^r} \otimes U
\]
given by projection onto $L_{\bbmu^r} \otimes U$.  This intertwiner is zero because $\bbmu^r$ is not of the form $\bbmu^r = \wtau$ for some $\tau$, giving the claim.
\end{proof}

\begin{prop} \label{prop:gt-factorization}
For any Gelfand-Tsetlin pattern
\[
\wbbmu = \{\wmu^1 \prec \wmu^2 \prec \cdots \prec \wmu^n = \wlambda\}
\]
subordinate to $\wlambda$, we have the factorization
\[
c(\wbbmu, \lambda) = \prod_{i = 1}^{n - 1} c(\mu^i, \mu^{i+1}).
\]
\end{prop}
\begin{proof}
By induction on $n$, it suffices to check that 
\[
c(\wbbmu, \lambda) = c(\mu, \lambda) c\Big(\{\wmu^1 \prec \cdots \prec \wmu^{n-1}\}, \mu^{n-1}\Big).
\]
Let $\mu = \mu^{n-1}$.  By Proposition \ref{prop:gt-lower}, the basis vector $v_{\wbbmu}$ lies in the $U_q(\gl_{n-1})$ submodule $L_{\wmu} \subset L_{\wlambda}$ with highest weight vector $v_{\gt(\wmu)}$.  Let $U \subset W_{k-1}$ be the $U_q(\gl_{n-1})$-submodule consisting of elements of weight $0$ under $q^{h_n}$.  Consider the $U_q(\gl_{n-1})$-intertwiner
\[
\phi: L_{\wmu} \to L_{\wmu} \otimes U
\]
given by composing $\wPhi_\lambda$ with the projection onto $L_{\wmu} \otimes U$.  The matrix element $c(\wbbmu, \lambda)$ lies in $U$, hence is the matrix element of $v_{\wbbmu}$ in $\phi$.  Notice that $\phi$ maps the $U_q(\gl_{n-1})$-highest weight vector $v_{\gt(\wmu)}$ to 
\[
c(\mu, \lambda) v_{\gt(\wmu)} \otimes w_{k-1} + (\text{l.o.t.})
\]
so that $\phi = c(\mu, \lambda) \wPhi_\mu$ and the matrix element of $v_{\wbbmu}$ is the desired
\[
c(\mu, \lambda)c\Big(\{\wmu^1 \prec \cdots \prec \wmu^{n-1}\}, \mu^{n-1}\Big). \qedhere
\]
\end{proof}

\subsection{Matrix elements as applications of Macdonald difference operators}

Our main technical result expresses matrix elements of $U_q(\gl_n)$-intertwiners as the application of Macdonald difference operators to an explicit kernel.  Define the elements $\Delta^{k-1}_1(\mu)$ and $\Delta^{k-1}_2(\mu)$ by
\begin{equation} \label{eq:delta}
\Delta^{k-1}_1(\mu) = \prod_{i < j} [\bmu_i - \bmu_j + (k - 1)]_{k-1} \qquad \Delta^{k-1}_2(\mu) = \prod_{i < j}[\bmu_i - \bmu_j - 1]_{k - 1}
\end{equation}
and the element $\Delta^{k-1}(\mu, \lambda)$ by
\begin{equation} \label{eq:delta-cross}
\Delta^{k-1}(\mu, \lambda) = \prod_{i \leq j} [\lambda_i - \mu_j + k (j - i) + k-1]_{k-1} \prod_{i < j}[\mu_i - \lambda_j + k (j - i) - 1]_{k-1}.
\end{equation}
We use Theorem \ref{thm:cg} to compute the diagonal matrix elements of $\wPhi_\lambda$ in terms of these elements, resulting in the following expression after manipulation.  We defer the proof of Proposition \ref{prop:diag-coeff} to Section \ref{sec:as}.

\begin{prop} \label{prop:diag-coeff}
Let $\mu' = \mu + (k-1)\one$, and $\nu' = \nu  + (k-1)\one$.  Then $c(\mu, \lambda)$ is given by
\begin{align*}
c(\mu, \lambda) &=  \frac{(-1)^{(n - 1)(k-1)} q^{(n - 1)k(k - 1)}}{\Delta_2^{k-1}(\lambda) \Delta_1^{k-1}(\mu)} \sum_{\bnu' = \bmu' - (k - 1)\one}^{\bmu'} (-1)^{|\bnu'| - |\bmu'|} q^{k(|\bnu'| - |\bmu'|)} \prod_i \frac{1}{[\bnu_i' - \bmu_i' + (k - 1)]![\bmu_i' - \bnu_i']!}\\
&\phantom{===} \frac{\prod_{i < j}  [\bmu_i' - \bmu_j' + k-1]_{2k-1}  \prod_{i < j} [\bnu_i' - \bnu_j']}{\prod_{i < j}[\bnu_i' - \bmu_j' + (k - 1)]_k [\bmu_i' - \bnu_j']_k} \prod_{i \leq  j} [\blambda_i - \bnu_j' + (k - 1)]_{k-1} \prod_{i < j} [\bnu_i' - \blambda_j - 1]_{k-1}.
\end{align*}
\end{prop}

In this form, we can now identify the matrix element with an application of Macdonald difference operators.

\begin{theorem} \label{thm:mat-elt}
Let $\mu' = \mu + (k - 1)\one$.  The matrix element $c(\mu, \lambda)$ is given by 
\[
c(\mu, \lambda) = \frac{\prod_{a = 1}^{k-1} D_{n-1, q^{2\bmu}}(q^{2a}; q^{-2}, q^{2(k-1)}) \Delta^{k-1}(\mu', \lambda)}{\Delta^{k-1}_1(\mu) \Delta^{k-1}_2(\lambda)},
\]
where $D_{n-1, q^{2\bmu}}(q^{2a}; q^{-2}, q^{2(k-1)})$ was defined in (\ref{eq:spec-op}).
\end{theorem}
\begin{proof}
For an expression $E$, let $1_E = 1$ if $E$ holds and $1_E = 0$ otherwise.  Interpreting $\Res_l(q^2)$ in the sense of Subsection \ref{sec:res-extend}, notice that
\begin{align*}
\Res_l&(q^2) D_{(n-1)l, q^{2\bmu}}(q^{l + 1}; q^{-2l}, q^2) \prod_{a = 0}^{l - 1}\prod_{i \leq j} [\blambda_i - \bmu_j^a + k/2] \prod_{i < j} [\bmu_i^a - \blambda_j - k/2]\\
&= \Res_l(q^2)\left(\sum_I (-1)^{(n - 1)l - |I|} q^{k((n - 1)l - |I|)}\prod_{(i, a) \in I; (j, b) \notin I} \frac{[\bmu_i^a - \bmu_j^b + 1]}{[\bmu_i^a - \bmu_j^b]}\right.\\
&\phantom{=================}  \left.\prod_{i \leq j, a} [\blambda_i - \bmu_j^a + l 1_{(j, a) \in I} + k/2] \prod_{i < j, a} [\bmu_i^a - \blambda_j - l 1_{(i, a) \in I} - k/2]\right)\\
&= \sum_I (-1)^{(n - 1)l - |I|} q^{k(l(n-1) - |I|)}\!\!\!\!\!\!\!\!\!\!  \prod_{(i, a) \in I; (j, b) \notin I}\!\!\!\!\!\! \frac{[\bmu_i - \bmu_j + a - b + 1]}{[\bmu_i - \bmu_j + a - b]}\\
&\phantom{================}  \prod_{i \leq j, a} [\blambda_i - \bmu_j - a + l 1_{(j, a) \in I} + l] \prod_{i < j, a} [\bmu_i + a - \blambda_j - l 1_{(i, a) \in I} - l],
\end{align*}
where both sums are over subsets $I \subset \{(i, a) \mid 1 \leq i \leq n - 1, 0 \leq a \leq l - 1\}$.  If $(i, a) \in I$ and $(i, a + 1) \notin I$, then the corresponding term is zero, so the only subsets $I$ which contribute to the sum are those of the form 
\[
I = \{(i, a) \mid a \geq s_i\} \text{ for some $s_1, \ldots, s_{n-1}$}.
\]
We rewrite the sum in these terms as
\begin{align*}
\Res_l(q^2)& D_{(n-1)l, q^{2\bmu}}(q^{l+1}; q^{-2l}, q^2)\prod_{a = 0}^{l - 1}\prod_{i \leq j} [\blambda_i - \bmu_j^a + k/2] \prod_{i < j} [\bmu_i^a - \blambda_j - k/2]\\
&= \!\!\!\!\sum_{s_1, \ldots, s_{n-1} = 0}^{l} \!\!\!\!(-1)^{- \sum_i s_i} q^{k \sum_i s_i} \prod_{i \leq j} [\blambda_i - \bmu_j + 2 l - s_i]_l \prod_{i < j} [\bmu_i - \blambda_j - l + s_i - 1]_l\\
&\phantom{============================} \prod_{i, j} \prod_{a = s_i}^{l - 1} \prod_{b = 0}^{s_j - 1} \frac{[\bmu_i - \bmu_j + a - b + 1]}{[\bmu_i - \bmu_j + a - b]}.
\end{align*}
Observe now that 
\begin{multline*}
\prod_{i, j} \prod_{a = s_i}^{l - 1} \prod_{b = 0}^{s_j - 1} \frac{[\bmu_i - \bmu_j + a - b + 1]}{[\bmu_i - \bmu_j + a - b]} =
\prod_{i, j}  \prod_{b = 0}^{s_j - 1} \frac{[\bmu_i - \bmu_j + l - b]}{[\bmu_i - \bmu_j + s_i - b]}\\
=  \prod_{i, j}  \frac{[\bmu_i - \bmu_j + l]_{s_j}}{[\bmu_i - \bmu_j + s_i]_{s_j}}
= \prod_{i < j} \frac{[\bmu_i - \bmu_j + l]_{2l + 1} [\bmu_i - \bmu_j + s_i - s_j]}{[\bmu_i - \bmu_j - s_j + l]_{l + 1} [\bmu_i - \bmu_j + s_i]_{l + 1}}\prod_i \frac{[l]!}{[s_i]![l - s_i]!}.
\end{multline*}
Substituting this into the previous expression and changing variables to $r_i = l - s_i$, we obtain
\begin{align*}
\Res_l&(q^2) D_{(n-1)l, q^{2\bmu}}(q^{l+1}; q^{-2l}, q^2) \prod_{a = 0}^{l - 1}\prod_{i \leq j} [\blambda_i - \bmu_j^a + k/2] \prod_{i < j} [\bmu_i^a - \blambda_j - k/2]\\
&= q^{kl (n - 1)}(-1)^{l(n-1)}\sum_{r_1, \ldots, r_{n-1} = 0}^{l} (-1)^{\sum_i r_i} q^{-k \sum_i r_i}\\
&\phantom{===}\prod_{i \leq j} [\blambda_i - \bmu_j + r_j + l]_l \prod_{i < j} [\bmu_i - \blambda_j - r_i - 1]_l \prod_{i < j} \frac{[\bmu_i - \bmu_j + l]_{2l + 1} [\bmu_i - \bmu_j - r_i + r_j]}{[\bmu_i - \bmu_j + r_j]_{l + 1} [\bmu_i - \bmu_j - r_i + l]_{l + 1}} \prod_i \frac{[l]!}{[r_i]![l - r_i]!},
\end{align*}
On the other hand, we have that
\begin{align*}
\Res_l(q^{2}) &\left(\prod_{a = 0}^{l - 1}\prod_{i \leq j} [\blambda_i - \bmu_j^a + k/2] \prod_{i < j} [\bmu_i^a - \blambda_j - k/2]\right)= \prod_{i \leq j} [\blambda_i - \bmu_j + k - 1]_{k - 1} \prod_{i < j} [\bmu_i - \blambda_j - 1]_{k-1}.
\end{align*}
Therefore, by Lemma \ref{lem:res-diff} and Corollary \ref{corr:daha-comm-shift} with $l = k - 1$ and $X_i^a = q^{2\bmu_i^a}$, we conclude that 
\begin{align*}
\prod_{a = 1}^l &D_{n - 1, q^{2\bmu}}(q^{2a}; q^{-2}, q^{2l}) \prod_{i \leq j} [\blambda_i - \bmu_j + k - 1]_{k - 1} \prod_{i < j} [\bmu_i - \blambda_j - 1]_{k-1}\\
&= q^{kl (n - 1)}(-1)^{l(n-1)}\sum_{r_1, \ldots, r_{n-1} = 0}^{l} (-1)^{\sum_i r_i} q^{-k \sum_i r_i}\\
&\phantom{===}\prod_{i \leq j} [\blambda_i - \bmu_j + r_j + l]_l \prod_{i < j} [\bmu_i - \blambda_j - r_i - 1]_l \prod_{i < j} \frac{[\bmu_i - \bmu_j + l]_{2l + 1} [\bmu_i - \bmu_j - r_i + r_j]}{[\bmu_i - \bmu_j + r_j]_{l + 1} [\bmu_i - \bmu_j - r_i + l]_{l + 1}} \prod_i \frac{[l]!}{[r_i]![l - r_i]!}.
\end{align*}
Dividing both sides by $\prod_{i < j} [\blambda_i - \blambda_j  - 1]_l \prod_{i \leq j} [\bmu_i - \bmu_j + l]_l$, we obtain
\begin{align*}
&\frac{\prod_{a = 1}^l D_{n - 1, q^{2\bmu}}(q^{2a}; q^{-2}, q^{2l}) \prod_{i \leq j} [\blambda_i - \bmu_j + k - 1]_{k - 1} \prod_{i < j} [\bmu_i - \blambda_j - 1]_{k-1}}{\prod_{i < j} [\blambda_i - \blambda_j  - 1]_l \prod_{i \leq j} [\bmu_i - \bmu_j + l]_l}\\
&\phantom{==========}=  \frac{q^{kl (n - 1)}(-1)^{l(n-1)}}{\prod_{i < j} [\blambda_i - \blambda_j  - 1]_l \prod_{i < j} [\bmu_i - \bmu_j + l]_l}\sum_{r_1, \ldots, r_{n-1} = 0}^{l} (-1)^{\sum_i r_i} q^{-k \sum_i r_i}\prod_i \frac{1}{[r_i]![l - r_i]!}\\
&\phantom{=============}\prod_{i \leq j} [\blambda_i - \bmu_j + r_j + l]_l \prod_{i < j} [\bmu_i - \blambda_j - r_i - 1]_l \prod_{i < j} \frac{[\bmu_i - \bmu_j + l]_{2l + 1} [\bmu_i - \bmu_j - r_i + r_j]}{[\bmu_i - \bmu_j + r_j]_{l + 1} [\bmu_i - \bmu_j - r_i + l]_{l + 1}},
\end{align*}
where the second expression is equal to $c(\mu - (k - 1)\one, \lambda)$ by Proposition \ref{prop:diag-coeff}.  Replacing $\mu$ by its shift $\mu' = \mu + (k - 1) \one$ and recalling the definitions of $\Delta^{k-1}_1(\mu)$, $\Delta^{k-1}_2(\lambda)$, and $\Delta^{k-1}(\mu, \lambda)$ yields the claimed expression
\[
c(\mu, \lambda) = \frac{\prod_{a = 1}^{k-1} D_{n-1, q^{2\bmu}}(q^{2a}; q^{-2}, q^{2(k-1)}) \prod_{i \leq j} [\blambda_i - \bmu_j' + k-1]_{k-1} \prod_{i < j}[\bmu_i' - \blambda_j - 1]_{k-1}}{\prod_{i \leq j} [\bmu_i' - \bmu_j' + k-1]_{k-1} \prod_{i < j}[\blambda_i - \blambda_j - 1]_{k-1}}. \qedhere
\]
\end{proof}

\section{Proving Macdonald's branching rule} \label{sec:branch}

We now put everything together to give a new proof of Macdonald's branching rule, which we reformulate for $t = q^k$ with $k$ a positive integer.    

\begin{theorem} \label{thm:branch-spec}
At $t = q^k$ for positive integer $k$, we have 
\[
P_\lambda(x_1, \ldots, x_n; q^2, q^{2k}) = \sum_{\mu \prec \lambda} x_n^{|\lambda| - |\mu|} P_\mu(x_1, \ldots, x_{n-1}; q^2, q^{2k}) \psi_{\lambda/\mu}(q^2, q^{2k})
\]
with 
\[
\psi_{\lambda/\mu}(q^2, q^{2k}) = \frac{\Delta^{k-1}(\mu, \lambda)}{\Delta^{k-1}_1(\mu) \Delta^{k-1}_2(\lambda)}.
\]
\end{theorem}

\begin{proof}
We induct on $n$. The base case is trivial because $P_\lambda(x_1; q^2, q^{2k}) = x_1^{|\lambda|}$.  For the inductive step, by Lemma \ref{lem:exclude}, it is enough to consider matrix elements for basis vectors of the form $v_{\wbbmu}$.  By Proposition \ref{prop:gt-factorization} and the inductive hypothesis, we thus have
\begin{align*}
\Tr(\wPhi_{\lambda}^n x^h) &= \sum_{\wmu^1 < \cdots < \wmu^{n - 1} < \wlambda} c(\mu^0, \mu^1) \cdots c(\mu^{n-1}, \lambda) \prod_i x_i^{(|\wmu^i| - |\wmu^{i-1}|)} \\
&= \sum_{\wmu < \wlambda} c(\mu, \lambda) x_n^{|\lambda| - |\mu| - (k-1)(n - 1)} \sum_{\wmu^1 < \cdots < \wmu^{n-2} < \wmu} c(\mu^0, \mu^1) \cdots c(\mu^{n-2}, \mu^{n-1}) \prod_{i = 1}^{n-1} x_i^{|\wmu^i| - |\wmu^{i-1}|} \\
&= \sum_{\wmu < \wlambda} c(\mu, \lambda) x_n^{|\lambda| - |\mu| - (k-1)(n - 1)} \Tr(\wPhi_\mu^{n-1} x^h) \\
&= \sum_{\wmu < \wlambda} c(\mu, \lambda) x_n^{|\lambda| - |\mu| - (k-1)(n - 1)} P_\mu(\ux; q^2, q^{2k}) \Tr(\wPhi_0^{n-1} x^h),
\end{align*}
where $\ux = (x_1, \ldots, x_{n-1})$.  By Corollary \ref{corr:denom}, we have that 
\[
\frac{\Tr(\wPhi_0^{n-1} x^h)}{\Tr(\wPhi_0^{n} x^h)} = (x_1 \cdots x_{n-1})^{k-1} x_n^{(k-1)(n-1)} \prod_{s = 1}^{k-1}\prod_{i = 1}^{n-1} (x_i - q^{2s} x_n)^{-1}.
\]
We conclude that 
\begin{align*}
\frac{\Tr(\wPhi_{\lambda}^n x^h)}{\Tr(\wPhi_0^{n} x^h)} &= (x_1 \cdots x_{n-1})^{k-1} \prod_{s = 1}^{k-1}\prod_{i = 1}^{n-1} (x_i - q^{2s} x_n)^{-1}\sum_{\wmu < \wlambda} c(\mu, \lambda) x_n^{|\lambda|} P_\mu(\ux/x_n; q^2, q^{2k})\\
&= (x_1 \cdots x_{n-1})^{k-1} \prod_{s = 1}^{k-1}\prod_{i = 1}^{n-1} (x_i - q^{2s} x_n)^{-1}\sum_{\mu = \lambda_\downarrow - (k-1)\one}^{\lambda^\uparrow} c(\mu, \lambda) x_n^{|\lambda|} P_\mu(\ux/x_n; q^2, q^{2k})\\
&= x_n^{(k-1)(n-1)}\prod_{s = 1}^{k-1}\prod_{i = 1}^{n-1} (x_i - q^{2s} x_n)^{-1}\sum_{\mu' = \lambda_\downarrow}^{\lambda^\uparrow + (k-1)\one} c(\mu' - (k-1)\one, \lambda) x_n^{|\lambda|} P_{\mu'}(\ux/x_n; q^2, q^{2k}),
\end{align*}
where $\lambda_\downarrow = (\lambda_2, \ldots, \lambda_{n})$ and $\lambda^\uparrow = (\lambda_1, \ldots, \lambda_{n-1})$ are vectors of lower and upper indices for $\mu$ so that $\sum_{\mu \prec \lambda} = \sum_{\mu = \lambda_\downarrow}^{\lambda^\uparrow}$ in the notation of (\ref{eq:summation}).  Note that $\wmu < \wlambda$ if and only if $\lambda_i \geq \mu_i \geq \lambda_{i+1} - (k - 1)$.  By the expression for $c(\mu' - (k-1)\one, \lambda)$ given in Theorem \ref{thm:mat-elt}, we obtain
\begin{align*}
P_\lambda(x; q^2, q^{2k}) = \frac{\Tr(\wPhi_{\lambda}^n x^h)}{\Tr(\wPhi_0^{n} x^h)} &=  x_n^{(k-1)(n-1)}\prod_{s = 1}^{k-1}\prod_{i = 1}^{n-1} (x_i - q^{2s} x_n)^{-1}\sum_{\mu' = \lambda_\downarrow}^{\lambda^\uparrow + (k-1)\one} x_n^{|\lambda|} P_{\mu'}(\ux/x_n; q^2, q^{2k})\\
&\phantom{===}\frac{\prod_{a = 1}^{k-1} D_{n-1, q^{2\bmu}}(q^{2a}; q^{-2}, q^{2(k-1)}) \prod_{i \leq j} [\blambda_i - \bmu_j' + k-1]_{k-1} \prod_{i < j}[\bmu_i' - \blambda_j - 1]_{k-1}}{\prod_{i \leq j} [\bmu_i' - \bmu_j' + k-1]_{k-1} \prod_{i < j}[\blambda_i - \blambda_j - 1]_{k-1}}.
\end{align*}
Define the operator 
\[
\wD_{n - 1, q^{2\bmu'}}(q^{2a}; q^2, q^{2k}) = \sum_r (-1)^{n - 1 - r} q^{2a(n - 1 - r)} \wD_{n - 1, q^{2\bmu'}}^r (q^2, q^{2k}),
\]
and note that it is diagonalized on $P_{\mu'}(\ux; q^2, q^{2k})$ by Proposition \ref{prop:sym-diag}.  Notice now that the function 
\[
\prod_{i \leq j} [\blambda_i - \bmu_j' + k-1]_{k-1} \prod_{i < j}[\bmu_i' - \blambda_j - 1]_{k-1}
\]
is $0$ for $\lambda_{i+1} - (k - 1) \leq \mu_i' < \lambda_{i+1}$ and $\lambda_i < \mu_i' \leq \lambda_i + (k - 1)$, so it is $(\lambda_\downarrow, \lambda^\uparrow, k - 1)$-adapted.  Applying Proposition \ref{prop:mac-adj} to this function yields
\begin{align*}
P_\lambda(x; q^2, q^{2k}) &=  x_n^{(k-1)(n-1)} \prod_{s = 1}^{k-1} \prod_{i = 1}^{n-1} (x_i - q^{2s} x_n)^{-1} \sum_{\mu' = \lambda_\downarrow}^{\lambda^\uparrow} x_n^{|\lambda|} \prod_{a = 1}^{k-1} \wD_{n - 1, q^{2\bmu'}}(q^{2a}; q^2, q^{2k}) P_{\mu'}(\ux/x_n; q^2, q^{2k}) \\
&\phantom{==========================}\frac{\prod_{i \geq j} [\blambda_j - \bmu_i' + k-1]_{k-1} \prod_{i < j} [\bmu_i' - \blambda_j - 1]_{k-1}}{\prod_{i \leq j} [\bmu_i' - \bmu_j' + k - 1]_{k-1} \prod_{i < j} [\blambda_i - \blambda_j - 1]_{k-1}}\\
&=  x_n^{(k-1)(n-1)} \prod_{s = 1}^{k-1} \prod_{i = 1}^{n-1} \frac{x_i / x_n - q^{2s}}{x_i - q^{2s} x_n}\sum_{\mu' = \lambda_\downarrow}^{\lambda^\uparrow} x_n^{|\lambda| - |\mu'|} P_{\mu'}(\ux; q^2, q^{2k})\\
&\phantom{==========================} \frac{\prod_{i \geq j} [\blambda_j - \bmu_i' + k-1]_{k-1} \prod_{i < j} [\bmu_i' - \blambda_j - 1]_{k-1}}{\prod_{i \leq j} [\bmu_i' - \bmu_j' + k - 1]_{k-1} \prod_{i < j} [\blambda_i - \blambda_j - 1]_{k-1}}\\
&= \sum_{\mu' \prec \lambda} x_n^{|\lambda| - |\mu'|} P_{\mu'}(\ux; q^2, q^{2k})\frac{\prod_{i \geq j} [\blambda_j - \bmu_i' + k - 1]_{k-1} \prod_{i < j} [\bmu_i' - \blambda_j - 1]_{k-1}}{\prod_{i \leq j} [\bmu_i' - \bmu_j' + k - 1]_{k-1} \prod_{i < j} [\blambda_i - \blambda_j - 1]_{k-1}},
\end{align*}
which is the desired result.
\end{proof}

\section{Specializing the expression for diagonal matrix elements} \label{sec:as}

We will prove Proposition \ref{prop:diag-coeff} using a result of \cite{AS} on reduced Clebsch-Gordan coefficients.  We normalize and translate this result to matrix elements to obtain our desired expression. We first modify the intertwiner slightly.  Consider the composition 
\[
\wPsi_\lambda: L_{\lambda + (k - 1) \wrho} \to L_{\lambda + (k - 1)\wrho} \otimes W_{k-1} \simeq L_{\lambda + (k - 1)\wrho - (k-1)\one} \otimes \Sym^{(k-1)n}\CC^n.
\]
The diagonal matrix element $c(\mu, \lambda)$ of $v_{\gt(\wmu)}$ in $\wPhi_\lambda$ is equal to the matrix element from $v_{\gt(\wmu)}$ to $v_{\gt(\wmu) - (k - 1) \one}$ in $\wPsi_\lambda$. We will compute this matrix element instead.

\subsection{The expression of \cite{AS} for reduced Clebsch-Gordan coefficients}

The $U_q(\gl_n)$-representation $L_\tau \otimes \Sym^p \CC^n$ contains each irreducible with multiplicity at most one, meaning that for any $\tau'$, there is a one-dimensional family of intertwiners
\[
L_{\tau'} \to L_\tau \otimes \Sym^p \CC^n.
\]
In \cite{AS}, a general formula for the matrix coefficients in the Gelfand-Tsetlin basis of one such map is given.  Such matrix coefficients are known as Clebsch-Gordan coefficients. 

\begin{remark}
Note that \cite{AS} uses the coproduct $\Delta_{AS} = \Delta^{21}$.  As bialgebras, $U_{q^{-1}}(\gl_n)$ equipped with $\Delta_{AS}$ and $U_q(\gl_n)$ equipped with $\Delta$ are isomorphic, so we state and apply here the formulas of \cite{AS} with $q$ and $q^{-1}$ exchanged.  
\end{remark}

Note that a Gelfand-Tsetlin basis vector $v_{\bbxi}$ for $\Sym^p\CC^n$ takes the form $\bbxi^i = (\xi^i, 0,\ldots, 0)$, so we will denote this by $\bbxi^i = \xi^i$.  For basis vectors $v_{\bbsigma} \in L_{\tau'}$ and $v_{\bbeta} \otimes v_{\bbxi} \in L_\tau \otimes \Sym^p \CC^n$, it is shown in \cite{AS} that the corresponding Clebsch-Gordan coefficient is given by a product
\[
\wC\left[\begin{matrix} \tau & p & \tau' \\ \bbeta & \bbxi & \bbsigma\end{matrix}\right] = \prod_{i = 1}^{n - 1}C\left[\begin{matrix} \bbsigma^{i+1} & \bbxi^{i+1} & \bbeta^{i+1} \\ \bbsigma^i & \bbxi^i & \bbeta^i \end{matrix}\right],
\]
where product is over reduced Clebsch-Gordan coefficients whose values are given by the following.  

\begin{theorem}[{\cite[Equation (3.4)]{AS}}] \label{thm:cg}
The reduced Clebsch-Gordan coefficient of the map 
\[
L_{\tau'} \to L_\tau \otimes \Sym^p \CC^n
\]
is given by 
\begin{align*}
C\left[\begin{matrix} \tau & p & \tau' \\ \eta & r & \eta' \end{matrix}\right] &= q^{-\frac{b}{2}} \frac{S(\eta', \eta) S(\tau, \eta) S(\tau', \tau') S(\eta, \eta)}{S(\tau', \tau) S(\tau', \eta')} [p - r]!^{1/2} \\
&\phantom{=================} \sum_{\sigma} (-1)^{|\sigma| - |\eta|} q^{(p - r + 1)(|\sigma| - |\eta|)} \frac{S(\sigma, \sigma)^2 S(\tau', \sigma)^2}{S(\sigma, \eta)^2 S(\eta', \sigma)^2 S(\tau, \sigma)^2},
\end{align*}
where the sum is over $\sigma$ of length $n - 1$ satisfying
\[
\max\{\eta_i, \tau_{i + 1}'\} \leq \sigma_i \leq \min\{\eta_i', \tau_i\}
\]
and where $b$ and $S$ are given by
\[
b = \sum_{i < j} (\tau_i' - \tau_i)(\tau_j' - \tau_j) - \sum_{i < j}(\eta_i' - \eta_i)(\eta_j' - \eta_j) + \sum_i (\eta_i' - \eta_i)(\eta_i - i + 1) - \sum_i (\tau_i' - \tau_i)(\tau_i - i + 1) + (p - r)(|\tau| - |\eta|)
\]
and
\[
S(a, b)^2 = \frac{\prod_{i \leq j} [a_i - b_j + j - i]!}{\prod_{i < j} [b_i - a_j + j - i - 1]!}.
\]
\end{theorem}

\subsection{Specializing the reduced Clebsch-Gordan coefficient}

We restrict now to our case of $\wPsi_\lambda$.  The relevant parameters are 
\[
\tau' = \wtilde{\lambda} \qquad \tau = \wtilde{\lambda} - (k - 1)\one \qquad \eta' = \wtilde{\mu} \qquad \eta = \wtilde{\mu} - (k - 1)\one \qquad p = n(k - 1) \qquad r = (n - 1)(k-1).
\]
In this case, we see that 
\begin{align*}
b &= \frac{n(n - 1)}{2} (k - 1)^2 - \frac{(n - 1)(n - 2)}{2} (k - 1)^2 + (k - 1)|\wtilde{\mu}|\\
&\phantom{====} - \frac{(n - 1)(n - 2)}{2} (k - 1) - (k - 1)|\wtilde{\lambda}| + \frac{n(n - 1)}{2}(k - 1) + (k - 1)(|\wtilde{\lambda}| - |\wtilde{\mu}|) \\
&= (n - 1)(k - 1)^2 + (n - 1)(k - 1)\\
&= (n - 1) k (k - 1).
\end{align*}
Further, the constraint on $\sigma$ takes the form 
\[
\max\{\mu_i - (k - 1) - (k - 1)i, \lambda_{i + 1} - (k - 1)(i + 1)\} \leq \sigma_i \leq \min\{\mu_i - (k - 1)i, \lambda_i - (k - 1) - (k - 1)i\},
\]
so if $\sigma = \nu + (k - 1)\wrho$, we have 
\[
\max\{\mu_i - (k - 1), \lambda_{i + 1} - (k - 1)\} \leq \nu_i \leq \min\{\mu_i, \lambda_i - (k - 1)\}.
\]
Translating and canceling a factor, we have that 
\begin{align*}
C\left[\begin{matrix} \wtilde{\lambda} - (k - 1)\one & n(k - 1) & \wtilde{\lambda} \\ \wtilde{\mu} - (k - 1)\one & (n - 1)(k - 1) & \wtilde{\mu} \end{matrix}\right] &= q^{-\frac{(n-1)k (k - 1)}{2}} [k - 1]!^{1/2} \frac{S(\wmu, \wmu - (k - 1)\one) S(\wlambda, \wlambda) S(\wmu, \wmu)}{S(\wlambda, \wlambda - (k - 1)\one)}\\
&\phantom{===} \sum_{\nu} (-1)^{|\nu| - |\mu|} q^{k(|\nu| - |\mu|)} \frac{S(\wnu, \wnu)^2 S(\wlambda, \wnu)^2}{S(\wnu, \wmu - (k - 1)\one)^2 S(\wmu, \wnu)^2 S(\wlambda - (k - 1)\one, \wnu)^2}.
\end{align*}
Denote the latter sum by $\Sigma(\mu, \lambda)$ and the prefactor by $B(\mu, \lambda)$.

\subsection{Computing the normalization factor}

Write $\underline{\lambda}$ for the truncation of $\lambda$, and note that both interpretations of $\underline{\wtilde{\lambda}}$ are equal. Further, denote by $\sgt(r)$ the pattern
\[
\{(k-1) \prec 2(k-1) \prec \cdots \prec r(k-1)\},
\]
where for $1 \leq i \leq r$, $i(k - 1)$ is identified with the length $i$ signature $(i(k -1), 0, \ldots, 0)$; note that $v_{\sgt(n)}$ has weight $(k - 1)\one$ in $\Sym^{n(k - 1)}\CC^n$.

We now consider the special case where $\mu = \underline{\lambda}$, which will allow us to translate between normalization factors for the Clebsch-Gordan coefficients.  In this case, the constraint on $\nu$ implies the sum is over the single term $\nu = \mu - (k - 1)\one = \underline{\lambda} - (k - 1)\one$, and the matrix coefficient is
\begin{align*}
C&\left[\begin{matrix} \wtilde{\lambda} - (k - 1)\one & n(k - 1) & \wtilde{\lambda} \\ \wtilde{\underline{\lambda}} - (k - 1)\one & (n - 1)(k - 1) & \wtilde{\underline{\lambda}} \end{matrix}\right]\\
 &\phantom{====}= q^{-\frac{(n-1)k(k-1)}{2}} [k - 1]!^{1/2} (-1)^{(n - 1)(k - 1)} q^{-k (k-1)(n-1)} \\
&\phantom{=========} \frac{S(\wulambda, \wulambda - (k - 1) \one) S(\wlambda, \wlambda) S(\wulambda, \wulambda) S(\wulambda, \wulambda)^2 S(\wlambda, \wulambda - (k - 1) \one)^2}{S(\wlambda, \wlambda - (k - 1)\one) S(\wulambda, \wulambda)^2 S(\wulambda, \wulambda - (k - 1)\one)^2 S(\wlambda - (k - 1)\one, \wulambda - (k - 1)\one)^2}\\
&\phantom{====}= (-1)^{(n - 1)(k - 1)} q^{-\frac{3(n-1)k (k - 1)}{2}} [k - 1]!^{1/2} \frac{S(\wlambda, \wlambda) S(\wulambda, \wulambda)S(\wlambda, \wulambda - (k - 1) \one)^2}{S(\wlambda, \wlambda - (k - 1)\one) S(\wulambda, \wulambda - (k - 1)\one) S(\wlambda, \wulambda)^2}.
\end{align*}
Notice now that 
\begin{align*}
\left(\frac{S(\wlambda, \wtau)S(\wulambda, \wutau)}{S(\wlambda, \wutau)^2}\right)^2 &= \prod_{1 \leq i \leq j \leq n - 1} \frac{[\blambda_i - \btau_j]!^2}{[\blambda_i - \btau_j]!^2} \prod_{1 \leq i < j \leq n - 1} \frac{[\btau_i - \blambda_j  - 1]!^2}{[\btau_i - \blambda_j - 1]!^2} \prod_{1 \leq i \leq n} [\blambda_i - \btau_n]! \prod_{1 \leq i < n} \frac{[\btau_i - \blambda_n - 1]!^2}{[\btau_i - \blambda_n - 1]!}\\
&= [\blambda_n - \btau_n]! \prod_{i = 1}^{n - 1} [\blambda_i - \btau_n]! [\btau_i - \blambda_n - 1]!.
\end{align*}
Applying this twice, we conclude that 
\begin{align*}
C\left[\begin{matrix} \wtilde{\lambda} - (k - 1)\one & n(k - 1) & \wtilde{\lambda} \\ \wtilde{\underline{\lambda}} - (k - 1)\one & (n - 1)(k - 1) & \wtilde{\underline{\lambda}} \end{matrix}\right]&=(-1)^{(n - 1)(k - 1)} q^{-\frac{3(n - 1)k (k - 1)}{2}} \left(\prod_{i = 1}^{n - 1} \frac{[\blambda_i - \blambda_n]![\blambda_i - \blambda_n - 1]!}{[\blambda_i - \blambda_n + (k - 1)]! [\blambda_i - \blambda_n - k]!}\right)^{1/2}\\
&=(-1)^{(n - 1)(k - 1)} q^{-\frac{3(n - 1)k (k - 1)}{2}} \left(\prod_{i = 1}^{n - 1} \frac{[\blambda_i - \blambda_n - 1]_{k-1}}{[\blambda_i - \blambda_n + (k - 1)]_{k - 1}}\right)^{1/2}.
\end{align*}
Iterating this, we find that the diagonal Clebsch-Gordan coefficient of the highest weight vector is 
\begin{align*}
\wC&\left[\begin{matrix} \wlambda - (k - 1)\one & n(k-1) & \wlambda \\ \gt(\underline{\wlambda} - (k - 1)\one) & \sgt(n-1) &  \gt(\underline{\wlambda})\end{matrix}\right] = (-1)^{\frac{n(n - 1)(k-1)}{2}} q^{-\frac{3 n (n - 1)k (k - 1)}{4}} \left(\prod_{i < j} \frac{[\blambda_i - \blambda_j - 1]_{k-1}}{[\blambda_i - \blambda_j + (k - 1)]_{k - 1}}\right)^{1/2},
\end{align*}
where we recall that $\gt$ was defined in (\ref{eq:gt-def}).

\subsection{Proof of Proposition \ref{prop:diag-coeff}}

We now put everything together to prove Proposition \ref{prop:diag-coeff}.  The diagonal Clebsch-Gordan coefficient of each highest weight vector for $U_q(\gl_{n-1})$ in the Gelfand-Tsetlin basis is
\begin{align*}
\wC&\left[\begin{matrix} \wlambda - (k - 1)\one & n(k - 1) & \wlambda \\ \gt(\wmu - (k - 1)\one) & \sgt(n-1) & \gt(\wmu)\end{matrix}\right]\\
&\phantom{=======}= C\left[\begin{matrix} \wtilde{\lambda} - (k - 1)\one & n(k - 1) & \wtilde{\lambda} \\ \wtilde{\mu} - (k - 1)\one & (n - 1)(k - 1) & \wtilde{\mu} \end{matrix}\right] \wC\left[\begin{matrix} \wmu - (k - 1)\one & (n - 1)(k - 1) & \wmu \\ \gt(\underline{\wmu} - (k - 1)) & \sgt(n-2) & \gt(\underline{\wmu})\end{matrix}\right]\\
&\phantom{=======}= B(\mu, \lambda)\Sigma(\mu, \lambda)  (-1)^{\frac{(n-1)(n - 2)(k-1)}{2}} q^{-\frac{3 (n-1) (n - 2)k (k - 1)}{4}} \prod_{i < j} \frac{[\bmu_i - \bmu_j - 1]_{k-1}^{1/2}}{[\bmu_i - \bmu_j + (k - 1)]_{k - 1}^{1/2}}.
\end{align*}
In terms of $\Delta^{k-1}_1$ and $\Delta^{k-1}_2$ from (\ref{eq:delta}), the matrix element of $\wPsi_\lambda$ and hence $\wPhi_\lambda$ we are interested in is 
\begin{align*}
c(\mu, \lambda) &=\wC\left[\begin{matrix} \wlambda - (k - 1)\one & n(k - 1) & \wlambda \\ \gt(\wmu - (k - 1)) & \sgt(n-1) & \gt(\wmu)\end{matrix}\right] \wC\left[\begin{matrix} \wlambda - (k - 1)\one & n(k-1) & \wlambda \\ \gt(\wlambda - (k - 1)\one) &  \sgt(n-1) & \gt(\wlambda)\end{matrix}\right]^{-1}\\
&= B(\mu, \lambda)\Sigma(\mu, \lambda) (-1)^{(n-1)(k-1)} q^{\frac{3(n-1)k(k-1)}{2}} \frac{\Delta^{k-1}_1(\lambda)^{1/2} \Delta^{k-1}_2(\mu)^{1/2}}{\Delta^{k-1}_2(\lambda)^{1/2} \Delta^{k-1}_1(\mu)^{1/2}}\\
&= (-1)^{(n - 1)(k-1)} q^{(n - 1)k(k-1)} [k - 1]!^{1/2}\frac{\Delta^{k-1}_1(\lambda)^{1/2} \Delta^{k-1}_2(\mu)^{1/2}}{\Delta^{k-1}_2(\lambda)^{1/2} \Delta^{k-1}_1(\mu)^{1/2}}
 \frac{S(\wmu, \wmu - (k - 1)\one) S(\wlambda, \wlambda) S(\wmu, \wmu)}{S(\wlambda, \wlambda - (k - 1)\one)} \Sigma(\mu, \lambda),
\end{align*}
where
\[
\Sigma(\mu, \lambda) = \sum_\nu (-1)^{|\nu| - |\mu|} q^{k(|\nu| - |\mu|)} X(\nu, \mu, \lambda)
\]
with 
\[
X(\nu, \mu, \lambda) =  \frac{S(\wnu, \wnu)^2 S(\wlambda, \wnu)^2}{S(\wnu, \wmu - (k - 1)\one)^2 S(\wmu, \wnu)^2 S(\wlambda - (k - 1)\one, \wnu)^2}.
\]
Observe that 
\begin{align*}
\frac{S(\wlambda, \wlambda)}{S(\wlambda, \wlambda - (k - 1)\one)} &= \left(\prod_{i \leq j} \frac{[\blambda_i - \blambda_j]!}{[\blambda_i - \blambda_j + (k - 1)]!} \prod_{i < j} \frac{[\blambda_i - \blambda_j - k]!}{[\blambda_i - \blambda_j - 1]!}\right)^{1/2}\\
&= [k - 1]!^{-n/2} \Delta^{k-1}_1(\lambda)^{-1/2} \Delta_2^{k-1}(\lambda)^{-1/2}
\end{align*}
and 
\begin{align*}
S(\wmu, \wmu - (k - 1)\one)S(\wmu, \wmu) &= \left(\frac{\prod_{i \leq j}[\bmu_i - \bmu_j + (k - 1)]! [\bmu_i - \bmu_j]!}{\prod_{i < j}[\bmu_i - \bmu_j - k]! [\bmu_i - \bmu_j - 1]!}\right)^{1/2}\\
&= [k - 1]!^{\frac{n - 1}{2}}\prod_{i < j} [\bmu_i - \bmu_j]  \Delta^{k-1}_1(\mu)^{1/2}\Delta^{k-1}_2(\mu)^{1/2}.
\end{align*}
We conclude that 
\begin{align*}
c(\mu, \lambda) &= (-1)^{(n - 1)(k-1)} q^{(n - 1)k(k - 1)}\prod_{i < j} [\bmu_i - \bmu_j]\frac{\Delta_2^{k-1}(\mu)}{\Delta^{k-1}_2(\lambda)} \Sigma(\mu, \lambda).
\end{align*}
We now notice that 
\begin{align*}
S(\wnu, \wmu - (k - 1)\one)^2 S(\wmu, \wnu)^2 &= \frac{\prod_{i \leq j} [\bnu_i - \bmu_j + (k - 1)]! [\bmu_i - \bnu_j]!}{\prod_{i < j} [\bmu_i - \bnu_j - k]! [\bnu_i - \bmu_j - 1]!}\\
&= \prod_i [\bnu_i - \bmu_i]![\bmu_i - \bnu_i]! \prod_{i < j} [\bnu_i - \bmu_j + (k - 1)]_k [\bmu_i - \bnu_j]_k
\end{align*}
and that 
\begin{align*}
\frac{S(\wlambda, \wnu)^2}{S(\wlambda - (k - 1)\one, \wnu)^2} &= \prod_{i \leq j} \frac{[\blambda_i - \bnu_j]!}{[\blambda_i - \bnu_j - (k - 1)]!} \prod_{i < j} \frac{[\bnu_i - \blambda_j + k - 2]!}{[\bnu_i - \blambda_j - 1]!} \\
&= \prod_i [\blambda_i - \bnu_i]_{k-1} \prod_{i < j} [\blambda_i - \bnu_j]_{k-1} [\bnu_i - \blambda_j + (k - 2)]_{k-1}.
\end{align*}
We conclude that 
\begin{align*}
X(\nu, \mu, \lambda) &= \frac{S(\wnu, \wnu)^2 S(\wlambda, \wnu)^2}{S(\wnu, \wmu - (k - 1)\one)^2 S(\wmu, \wnu)^2 S(\wlambda - (k - 1)\one, \wnu)^2}\\
&= \prod_{i < j} [\bnu_i - \bnu_j] \prod_i \frac{[\blambda_i - \bnu_i]_{k-1}}{[\bnu_i - \bmu_i]![\bmu_i - \bnu_i]!} \prod_{i < j} \frac{[\blambda_i - \bnu_j]_{k-1} [\bnu_i - \blambda_j + (k - 2)]_{k-1}}{[\bnu_i - \bmu_j + (k - 1)]_k [\bmu_i - \bnu_j]_k}.
\end{align*}
Putting everything together, the expression we obtain for the diagonal matrix element is 
\begin{align*}
c(\mu, \lambda) &= (-1)^{(n - 1)(k-1)} q^{(n - 1)k(k - 1)} \prod_{i < j} [\bmu_i - \bmu_j]\frac{\Delta_2^{k-1}(\mu)}{\Delta^{k-1}_2(\lambda)} \sum_{\nu = \mu - (k - 1)\one}^\mu (-1)^{|\nu| - |\mu|} q^{k(|\nu| - |\mu|)}\\
&\phantom{===}  \prod_{i < j} [\bnu_i - \bnu_j] \prod_i \frac{[\blambda_i - \bnu_i]_{k-1}}{[\bnu_i - \bmu_i]![\bmu_i - \bnu_i]!}\prod_{i < j} \frac{[\blambda_i - \bnu_j]_{k-1} [\bnu_i - \blambda_j + (k - 2)]_{k-1}}{[\bnu_i - \bmu_j + (k - 1)]_k [\bmu_i - \bnu_j]_k}.
\end{align*}
In terms of $\mu'$ and $\nu'$, this is the desired
\begin{align*}
c(\mu, \lambda) &= (-1)^{(n - 1)(k-1)} q^{(n - 1)k(k - 1)} \prod_{i < j} [\bmu_i' - \bmu_j']\frac{\Delta_2^{k-1}(\mu')}{\Delta^{k-1}_2(\lambda)} \sum_{\nu' = \mu' - (k - 1)\one}^{\mu'} (-1)^{|\nu'| - |\mu'|} q^{k(|\nu'| - |\mu'|)}\\
&\phantom{===}  \prod_{i < j} [\bnu_i' - \bnu_j' + k (j - i)] \prod_i \frac{[\blambda_i - \bnu_i' + (k - 1)]_{k-1}}{[\bnu_i' - \bmu_i' + (k - 1)]![\bmu_i' - \bnu_i']!} \prod_{i < j} \frac{[\blambda_i - \bnu_j' + (k - 1)]_{k-1} [\bnu_i' - \blambda_j - 1]_{k-1}}{[\bnu_i' - \bmu_j' + (k - 1)]_k [\bmu_i' - \bnu_j']_k} \\
&=  \frac{(-1)^{(n - 1)(k-1)} q^{(n - 1)k(k - 1)}}{\Delta_2^{k-1}(\lambda) \Delta_1^{k-1}(\mu')} \sum_{\nu' = \mu' - (k - 1)\one}^{\mu'} (-1)^{|\nu'| - |\mu'|} q^{k(|\nu'| - |\mu'|)} \prod_i \frac{1}{[\bnu_i' - \bmu_i' + (k - 1)]![\bmu_i' - \bnu_i']!}\\
&\phantom{===} \frac{\prod_{i < j}  [\bmu_i' - \bmu_j' + k-1]_{2k-1}  \prod_{i < j} [\bnu_i' - \bnu_j']}{\prod_{i < j}[\bnu_i' - \bmu_j' + (k - 1)]_k [\bmu_i' - \bnu_j']_k} \prod_{i \leq  j} [\blambda_i - \bnu_j' + (k - 1)]_{k-1} \prod_{i < j} [\bnu_i' - \blambda_j - 1]_{k-1}.
\end{align*}

\bibliographystyle{alpha}
\bibliography{mac-bib}

\end{document}